\documentclass[
10pt
]{article}

\usepackage{color}
\definecolor{LinkColor}{rgb}{0,0,1}
\definecolor{lbcolor}{rgb}{0.85,0.85,0.85}
\definecolor{FrameColor}{rgb}{0.85,0.85,0.85}

\usepackage[ngerman, french, english]{babel}
\usepackage[utf8]{inputenc}
\usepackage[T1]{fontenc}
\usepackage{enumerate}
\usepackage{setspace}

\def\pskip{\\[-3mm]}
\def\bpskip{\\[-2mm]}
\usepackage{multicol}
\usepackage[babel,french=guillemets,german=swiss]{csquotes}

\usepackage{array}
\usepackage{booktabs}
\usepackage{framed}
\usepackage{rotating}
\usepackage{longtable}
\usepackage{multirow}
\usepackage{tabularx}
\newcolumntype{L}[1]{>{\raggedright\arraybackslash}p{#1}} % linksb?ndig mit Breitenangabe
\newcolumntype{C}[1]{>{\centering\arraybackslash}p{#1}} % zentriert mit Breitenangabe
\newcolumntype{R}[1]{>{\raggedleft\arraybackslash}p{#1}} % rechtsb?ndig mit Breitenangabe

\usepackage{graphicx}

\usepackage{amsmath}
\usepackage{amssymb}
\usepackage{dsfont}
\usepackage{nicefrac}
\usepackage{empheq}
\allowdisplaybreaks

\usepackage[%
pdftitle={Titel},% Titel der Diplomarbeit
pdfauthor={Autor},% Autor(en)
pdfcreator={LaTeX, LaTeX with hyperref and KOMA-Script},% Genutzte Programme
pdfsubject={Betreff}, % Betreff
pdfkeywords={Keywords}
]{hyperref} 

\hypersetup{colorlinks=true,% Definition der Links im PDF File
	linkcolor=LinkColor,%
	anchorcolor=LinkColor,%
	citecolor=LinkColor2,%
	filecolor=LinkColor,%
	menucolor=LinkColor,%
	pagecolor=LinkColor,%
	urlcolor=LinkColor,%
	allcolors=LinkColor}

\usepackage[savemem]{listings}
\usepackage{paralist}
{\begin{list}{$\diamondsuit$}{}}%
	{\end{list}}

\usepackage{amsthm}

\newtheoremstyle{tstyle}% name
{15pt}	% Space above
{5pt}	% Space below 
{\itshape}	% Body font
{}	% Indent amount: Indent amount: empty = no indent, \parindent = normal paragraph indent
{\bfseries}	% Theorem head font
{.}	% Punctuation after theorem head
{0.5em}	% Space after theorem head: Space after theorem head: { } = normal interword space; \newline = linebreak
{}	% Theorem head spec (can be left empty, meaning `normal')

\theoremstyle{tstyle}

\newtheorem{thm}{Theorem}
\newtheorem{lem}[thm]{Lemma}
\newtheorem{prop}[thm]{Proposition}
\newtheorem{cor}[thm]{Corollary}
\newtheorem{defn}[thm]{Definition}

\newtheoremstyle{cstyle}% name
{15pt}	% Space above
{5pt}	% Space below 
{}	% Body font
{}	% Indent amount: Indent amount: empty = no indent, \parindent = normal paragraph indent
{\bfseries}	% Theorem head font
{}	% Punctuation after theorem head
{0.2222em}	% Space after theorem head: Space after theorem head: { } = normal interword space; \newline = linebreak
{}	% Theorem head spec (can be left empty, meaning `normal')

\theoremstyle{cstyle}

\newtheorem*{com}{Comment}

\makeatletter
\g@addto@macro{\thm@space@setup}{\thm@headpunct{}}
\renewenvironment{proof}[1][\proofname]{\par
	\pushQED{\qed}%
	\normalfont \topsep6\p@\@plus6\p@\relax
	\trivlist
	\item[\hskip\labelsep
	\bfseries
	#1\@addpunct{\,}]\ignorespaces
}{%
	\popQED\endtrivlist\@endpefalse
}
\makeatother

\makeatletter
\g@addto@macro{\thm@space@setup}{\thm@headpunct{}}
\newenvironment{sketch-proof}[1][Sketch of the proof]{\par
	\pushQED{\qed}%
	\normalfont \topsep6\p@\@plus6\p@\relax
	\trivlist
	\item[\hskip\labelsep
	\bfseries
	#1\@addpunct{\,}]\ignorespaces
}{%
	\popQED\endtrivlist\@endpefalse
}
\makeatother

\def\RR{\mathbb R}
\def\VV{{\mathcal{V}}}
\def\NN{\mathbb N}

\def\BB{{\mathbb B_K}}

\def\IBB{\mathring{\mathbb B}_{K}}

\def\IBBB{\mathring{\mathbb B}_{2K}}

\def\eps{\varepsilon}
\def\supp{\textnormal{supp\,}}
\def\ee{\textnormal{e}}

\def\BR{{B_R(0)}}
\def\Br{{B_r(0)}}
\def\Bro{{B_{r_0}(0)}}
\def\BRZ{{B_{R_Z}(0)}}

\def\ddt{\frac{\mathrm d}{\mathrm dt}}

\def\dtau{\;\mathrm d\tau}
\def\dz{\;\mathrm dz}

\def\dy{\;\mathrm dy}
\def\dv{\;\mathrm dv}
\def\dt{\;\mathrm dt}
\def\ds{\;\mathrm ds}
\def\dtx{\;\mathrm d(t,x)}

\def\dtxv{\;\mathrm d(t,x,v)}
\def\dyw{\;\mathrm d(y,w)}

\def\B{{\bar B}}

\def\delzi{\partial_{z_i}}
\def\delzj{\partial_{z_j}}
\def\delxi{\partial_{x_i}}
\def\delxj{\partial_{x_j}}

\def\delvi{\partial_{v_i}}

\def\delt{\partial_{t}}
\def\delx{\partial_{x}}
\def\delv{\partial_{v}}
\def\delz{\partial_{z}}

\def\stimes{{\hspace{-0.01cm}\times\hspace{-0.01cm}}}
\def\scdot{{\hspace{1pt}\cdot\hspace{1pt}}}

\def\laplace{\Delta}

\def\MA{\mathbf A}
\def\MB{\mathbf B}
\def\MC{\mathbf C}

\def\Ma{\mathbf a}
\def\Mb{\mathbf b}

\def\Mf{\mathring{\mathbf f}}

\def\Mchi{\boldsymbol \chi}

\def\tand{\quad\text{and}\quad}
\def\twith{\quad\text{with}\quad}

\def\B*{{\bar B}}
\def\u*{{\bar u}}
\def\f*{f_{\u*}}
\def\g*{g_{\u*}}

\def\tf{{\tilde f}}

\def\tg{{\tilde g}}

\def\tZ{{\tilde Z}}
\def\tB{{\tilde B}}

\def\interior{{\textnormal{int\,}}}

\def\Lag{{\mathcal{L}}}

\def\wto{\rightharpoonup}

\def\itema{\item[\textnormal{(a)}]}
\def\itemb{\item[\textnormal{(b)}]}

\newcommand{\Underset}[3][0pt]{\ensuremath{\underset{\raise#1\hbox{\small\ensuremath{#2}}}{#3}}}

\begin{document}

\begin{center}	
	\LARGE{Optimal control of a Vlasov-Poisson plasma by an external magnetic field \\[2mm] Analysis of a \\tracking type optimal control problem}\\[5mm]
	\normalsize{P. Knopf}\\
	\textit{University of Bayreuth, 95440 Bayreuth, Germany}
	\texttt{Patrik.Knopf@uni-bayreuth.de}
\end{center}

\begin{abstract}
	%	We consider the three dimensional Vlasov-Poisson system in the plasma physical case. It describes the time evolution of the distribution function of a large number of electrically charged particles which move under the influence of a self-consistent electric field that is generated by the charge of the particles. The aim of various concrete applications is to control the distribution function of the plasma in a desired way by an external magnetic field that interacts with the particles via Lorentz force. This can be modeled as an optimal control problem that is a problem of variational calculus.
	In the paper \textit{Optimal control of a Vlasov-Poisson plasma by an external magnetic field - The basics for variational calculus} \cite{knopf} we have already introduced a set of admissible fields and we have proved that each of those fields induces a unique strong solution of the Vlasov-Poisson system. We have also established that the field-state operator that maps any admissible field onto its corresponding solution is continuous and weakly compact. In this paper we will show that this operator is also Fréchet differentiable and we will continue to analyze the optimal control problem that was introduced in \cite{knopf}. More precisely, we will establish necessary and sufficient conditions for local optimality and we will show that an optimal solution is unique under certain conditions.\\
	
\end{abstract}

\section{Introduction}

This paper is a sequel of \textit{Optimal control of a Vlasov-Poisson plasma by an external magnetic field - The basics for variational calculus} \cite{knopf}. It is recommended to read this paper previously. However, we will briefly sketch the main results of \cite{knopf} here: We consider the three dimensional Vlasov-Poisson system in the plasma physical case that is equipped with an external field $B$:
\begin{equation}
\label{VP}
\left\{
\begin{aligned}
&\partial_t f + v\cdot \partial_x f - \partial_x \psi \cdot \partial_v f + (v\times B)\cdot \delv f= 0, \quad f\big\vert_{t=0}= \mathring f,\\
&\psi_f(t,x) = \int \frac{\rho_f(t,y)}{|x-y|} \dy, \quad 
\rho_f(t,x) = \int f(t,x,v)\ \mathrm dv.
\end{aligned}
\right.
\end{equation}
This system describes the time evolution of the distribution function\linebreak $f=f(t,x,v)\ge 0\; \big(x,v\in\RR^3, i.e., z=(x,v)\in\RR^6\big)$ of a plasma whose ions move under the influence of a self-consistent electric field $-\delx\psi_f=-\delx\psi_f(t,x)$. Thereby we assume that $\mathring f\in C^2_c(\RR^6;\RR_0^+)$ is a fixed initial datum. The external magnetic field $B$, that interacts with the particles via Lorentz force $(v\times B)$, acts as a control in this model.

In \cite{knopf} we have already introduced a set of fields that are suitable for our approach. For any final time $T>0$ and any exponent $\beta>3$ let $\VV$ denote the Banach space $L^2(0,T;W^{2,\beta}(\RR^3;\RR^3)) \cap L^2(0,T;H^1(\RR^3;\RR^3))$ and let $\|\cdot\|_\VV$ denote its standard norm. Then, for any radius $K>0$, the closed ball
$$\BB := \big\{ B\in\VV \;\big\vert\; \|B\|_\VV \le K \big\}$$
is referred to as the \textbf{set of admissible fields}. Note that $\BB \subset L^2\big(0,T;C^{1,\gamma}\big)$ where $C^{1,\gamma}$ denotes the Hölder space with exponent $\gamma = 1- \nicefrac 3 \beta$. We have proved that any admissible field $B\in \BB$ induces a unique strong solution 
$$f_B\in W^{1,2}\big(0,T;C_b(\RR^6)\big) \cap C\big([0,T];C^1_b(\RR^6)\big) \cap L^\infty\big(0,T;W^{2,\beta}(\RR^6) \big)$$ 
of the initial value problem \eqref{VP}, i.e., $f_B$ satisfies \eqref{VP} almost everywhere and for all $t\in[0,T]$, $\supp f_B(t)$ is contained in some ball $B_R(0)$ for some radius $R>0$ depending only on $\mathring f$, $T$, $K$ and $\beta$. Moreover, $f_B$ preserves the $p$-norm, i.e., for all $t\in[0,T]$ and any $1\le p\le \infty$ it holds that $\|f_B(t)\|_{L^p} = \|\mathring f\|_{L^p}$.\pskip

Now it was possible to define the \textbf{field-state operator}
$$ f.\,:\, \BB \to C\big([0,T];L^2(\RR^6)\big), \quad B\mapsto f_B\;. $$
We could show that there exist positive constants $C_1,C_2,C_3,L_1,L_2,L_3$ depending only on $\mathring f$, $T$, $K$ and $\beta$ such that for all $B,H\in\BB$ the corresponding solutions $f_B$ and $f_H$ satisfy
\begin{align}
\label{HCF}
\left\{
\begin{aligned}
&\|f_B - f_H\|_{C([0,T];C_b)} \le L_1\, \|B-H\|_\VV \,,&& \|\delz f_B\|_{C([0,T];C_b)} \le C_1\\
&\|\delz f_B - \delz f_H\|_{C([0,T];C_b)} \le L_2\, \|B-H\|_\VV^\gamma \,,&& \|\delt f_B\|_{C([0,T];C_b)} \le C_2\\
&\|\delt f_B - \delt f_H\|_{L^2(0,T;C_b)} \le L_3\, \|B-H\|_\VV^\gamma \,, && \|D_z^2 f_B\|_{C([0,T];C_b)} \le C_3.
\end{aligned}
\right.
\end{align}
where $\delz=\partial_{(x,v)}$ denotes the gradient in phase space. This means that the field-state operator is Lipschitz continuous. Moreover, we have already established the following result: Let $(B_k)\subset \BB$ be weakly convergent in $\VV$ with limit $B\in\BB$. Then
\begin{align*}
f_{B_k} \wto f_B \quad\text{in}\; W^{1,2}\big(0,T;L^2(\RR^6)\big), \quad k\to\infty
\end{align*}
which means weak compactness of the field-state operator as the set of admissible controls is weakly (sequentially) compact. Note that \cite[Prop.\,15]{knopf} provides even more similar compactness results but only the above will be used in the later approach.\pskip

With this knowledge we have started to analyze an optimal control problem that we will also consider in this paper. The aim is to control the time evolution of the distribution function in such a way that its value at time $T$ matches a desired distribution function $f_d\in C^2_c(\RR^6)$ as closely as possible. More precisely we want to find a magnetic field $B$ such that the $L^2$-difference $\|f_B(T)-f_d\|_{L^2}$ becomes as small as possible. Therefore, we intend to minimize the quadratic cost functional\vspace{-1mm}
\begin{align}
\label{OP}
J(B) = \frac 1 2 \|f_B(T)-f_d\|_{L^2(\RR^6)}^2 + \frac \lambda 2  \|D_x B\|_{L^2([0,T]\times\RR^3;\RR^{3\times 3})}^2
\end{align}
subject to $B\in\BB$. Here $\lambda$ is a nonnegative parameter. The field $B$ is the control in this model. Since $\|f(t)\|_{p} = \|\mathring f\|_{p}$ for all $1\le p\le \infty$, $t\in[0,T]$ it makes sense to assume that $\|f_d\|_{p} = \|\mathring f\|_{p}$ for all $1\le p\le \infty$ because otherwise the exact matching $f(T)=f_d$ would be impossible. \pskip

In \cite[Thm.\,16]{knopf} we have established that this optimization problem has at least one globally optimal solution. As the field-state operator is nonlinear there is no reason to assume that it is convex. Thus, this result does not provide uniqueness of this optimal solution. Of course, it is also possible that there are several locally optimal solutions. Therefore, in Section 4, we will analyze the optimization problem \eqref{OP} with respect to the following topics:
\begin{itemize}
	\item Necessary conditions of first order for local optimality,\vspace{-1mm}
	\item derivation of an optimality system,\vspace{-1mm}
	\item sufficient conditions of second order for local optimality,\vspace{-1mm}
	\item uniqueness of the optimal control under certain conditions.
\end{itemize}
The methods we are using are oriented towards the approach by F. Tröltzsch in \cite{troeltzsch}. For this procedure we will need Fréchet differentiability of the field-state operator that will be established in Section 3. As the Fréchet derivative is a linear approximation we will find out that it is given by a linear inhomogenous Vlasov equation. These general Vlasov equations will be analyzed in Section~2. \vspace{-7mm}

\section{A general inhomogenous linear Vlasov equation}

%Since the Fr\'echet derivative is a linear approximation of the field-state operator at some certain point $B\in\BB$ it turns out that this derivative is determined by an inhomogenous linear Vlasov equation. In the following section we will analyze those linear Vlasov equations in general, i.e., we will establish some existence and uniqueness results. The type and the regularity of the solution will depend on the regularity of the coefficients.

Let $r_0\ge 0$ and $r_2>r_1 \ge 0$ be arbitrary. We consider the following inhomogenous linear version of the Vlasov equation:
\begin{align}
\label{LVL}
\delt f + v\scdot\delx f + \MA\scdot\delv f + (v\stimes \MB )\scdot\delv f = \delx\psi_f\scdot\MC + \Mchi\,\Phi_{\Ma,f} + \Mb,\;\; 
f\big\vert_{t=0} = \Mf
\end{align}
The coefficients are supposed to have the following regularity
\begin{align}
&\label{RCa} \Ma=\Ma(t,x,v) \in C\big([0,T];C^1_b(\RR^6)\big),	\\
&\label{RCb} \Mb=\Mb(t,x,v) \in C\big([0,T];C^1_b(\RR^6)\big),	\\
&\label{RCf} \Mf=\Mf(x,v) \in C^2_c(\RR^6), \\
&\label{RCA} \MA=\MA(t,x)  \in C\big([0,T];C^{1,\gamma}(\RR^3;\RR^3)\big),	\\
&\label{RCB} \MB=\MB(t,x)  \in C\big([0,T];C^{1,\gamma}(\RR^3;\RR^3)\big),	\\
&\label{RCC} \MC=\MC(t,x,v) \in C\big(0,T;C^1_b(\RR^6;\RR^3)\big), \\
&\label{RCchi} \Mchi=\Mchi(x,v) \in C^1_c(\RR^6;[0,1])
\end{align}
with\vspace{-3mm}
\begin{gather}
\label{SC1} \supp \Ma(t),\; \supp \Mb(t),\; \supp\Mf,\; \supp \MC(t) \subset \Bro,\quad t\in[0,T],\\
\label{SC2} \Mchi = 1 \;\;\text{on}\;\; B_{r_1}(0),\quad \supp \Mchi  \subset B_{r_2}(0)
\end{gather}
Moreover $\Phi_{\Ma,f}$ is given by
\begin{align}
\label{DEFPHI}
\Phi_{\Ma,f}(t,x) := -\iint \frac{x-y}{|x-y|^3}\cdot\delv \Ma(t,y,w) \, f(t,y,w) \;\mathrm dw \mathrm dy 
\end{align}
for all $(t,x)\in[0,T]\times\RR^3$. We will also use the notation
\begin{align}
\label{DEFPHI2}
\Phi_{\Ma,f}'(t,x) := - \iint \frac{x-y}{|x-y|^3}\cdot\Big( \delv \Ma \, \delx f- \delv f \, \delx \Ma \Big)(t,y,w) \;\mathrm dw \mathrm dy\;.
\end{align}
for $(t,x)\in[0,T]\times\RR^3$. Note that
\begin{align*}
\Phi_{\Ma,f} = \sum_{i=1}^3 \delxi \psi_{\delvi \Ma f} \tand \big[\Phi_{\Ma,f}'\big]_j = \sum_{i=1}^3 \delxi \psi_{\delvi \Ma\,\delxj f - \delvi f\,\delxj \Ma},\quad j=1,2,3.
\end{align*}
As $\Ma \in C\big([0,T];C^1_b(\RR^6)\big)$ with compact support $\supp \Ma(t) \subset \Bro$ for all ${t\in[0,T]}$, \cite[Lem.\,2]{knopf} provides the following inequalities: For any $r>0$ there exists some constant $c>0$ that may depend only on $r$ and $r_0$ such that for almost all $t\in[0,T]$,
\begin{align}
&\label{ESTPHI1}	\|\Phi_{\Ma,f}(t)\|_{L^2(B_r(0))} \le c\, \|\delv a(t)\|_\infty\, \|f(t)\|_{L^2(\Bro)}, && f\in L^2(0,T;L^2),\\
&\label{ESTPHI2}	\|\Phi_{\Ma,f}'(t)\|_{L^2(B_r(0))} \le c\, \|\delz a(t)\|_\infty\, \|\delz f(t)\|_{L^2(\Bro)}, \hspace{-2mm}&& f\in L^2(0,T;H^1),\\
&\label{ESTPHI3}	\|\Phi_{\Ma,f}(t)\|_{L^\infty} \le c\, \|\delv a(t)\|_\infty\, \|f(t)\|_{L^\infty(\Bro)}, && f\in L^2(0,T;L^\infty),\\
&\label{ESTPHI4}	\|\Phi_{\Ma,f}'(t)\|_{L^\infty} \le c\, \|\delz a(t)\|_\infty\, \|\delz f(t)\|_{L^\infty(\Bro)}, && f\in L^2(0,T;W^{1,\infty}).
\end{align}
If $\Ma \in C\big([0,T];C^2_b(\RR^6)\big)$ and $f\in C\big([0,T];C^1_b(\RR^6)\big)$ then $\Phi_{\Ma,f}$ is continuously differentiable with respect to $x$ with
\begin{align*}
\delxj \Phi_{\Ma,f}(t,x) &= \sum_{i=1}^3 \delxj \delxi \psi_{\delvi \Ma f} 
%= \sum_{i=1}^3 \delxi \psi_{\delvi \Ma \delxj f + \delvi\delxj \Ma f} \\& 
= \sum_{i=1}^3 \delxi \psi_{\delvi \Ma \delxj f - \delxj \Ma \delvi f} = \big[\Phi_{\Ma,f}'\big]_j(t,x)
\end{align*}
for all $(t,x)\in [0,T]\times\RR^3$. Because of density this result holds true if \linebreak$\Ma \in C\big([0,T];C^1_b(\RR^6)\big)$. If merely $f\in L^2(0,T;H^1)$ the result holds true in the weak sense. 

\smallskip

%
%	LEMMA
%

\begin{lem}
	\label{LVLZ}
	\hypertarget{HLVLZ}
	Let $A,B\in C\big([0,T];C^1_b(\RR^3;\RR^3)\big)$ be arbitrary. For any ${t\in[0,T]}$ and $z\in\RR^6$ the characteristic system
	\begin{align*}
	%\label{CSL}
	\dot x = v, \quad
	\dot v = \MA(s,x)+v\times \MB(t,x)\;,
	\end{align*}
	has a unique solution $Z\in C^1([0,T]\times[0,T]\times\RR^6;\RR^6)$, $Z(s,t,z)=(X,V)(s,t,z)$ to the initial value condition $Z(t,t,z)=z$. For any $r>0$ and all $s,t\in [0,T]$,
	$$Z(s,t,B_r(0)) \subset B_{\zeta(r)}(0) \twith \zeta(r) :=  \ee^{2T} \big(r + \sqrt{T} \|\MA\|_{L^2(0,T;L^\infty)} \big) \;.$$
	Moreover, there exists some constant $C(r)>0$ depending only on $\|\MA\|_{L^2(0,T;C^1_b)}$, $\|\MB\|_{L^2(0,T;C^1_b)}$ and $r$ such that for all $s,t\in[0,T]$,
	\begin{align*}
	\|\delz Z(s,t,\cdot)\|_{L^\infty(B_r(0))} \le C(r) \tand \|\delt Z(s,t,\cdot)\|_{L^\infty(B_r(0))} \le C(r)\;.
	\end{align*}
\end{lem}

The proof is simple and very similar to the proof of \cite[Lem.\,8]{knopf}. Therefore it will not be presented.\pskip

Now we can establish an existence and uniqueness result for classical solutions of the system \eqref{LVL} if the regularity conditions \eqref{RCa}-\eqref{RCchi} hold. Unfortunately the coefficients of the systems that will occur in this paper do not satisfy those strong conditions. However, we will still be able to prove an existence and uniqueness result for strong solutions of \eqref{LVL} if the regularity conditions are slightly weaker.

%
%	PROPOSITION
%

\begin{prop}
	\label{CSLVL}
	\hypertarget{HCSLVL}
	Suppose that the coefficients of the system \eqref{LVL} satisfy the regularity conditions \eqref{RCa}-\eqref{RCchi} and the support conditions \eqref{SC1},\eqref{SC2}. Then the initial value problem \eqref{LVL} has a unique classical solution ${f\in C^1([0,T]\times\RR^6)}$. Moreover for all $t\in[0,T]$, $\supp f(t) \subset B_{\zeta(r+1)}(0)$ with $r=\max\{r_0,r_2\}$ and $f$ is implicitely given by\vspace{-2mm}
	\begin{align}
	\label{EXPLVL}
	f(t,z) = \Mf\big( Z(0,t,z) \big) + \int\limits_0^t \big[\delx\psi_f\cdot\MC + \Mchi\Phi_{\Ma,f} + \Mb\big]\big(s,Z(s,t,z) \big)  \ds
	\end{align}
	for any $t\in[0,T],\;z\in\RR^6$.
	Moreover, there exists some constant $C>0$ depending only on $T$, $r_0$, $r_2$ and the standard norms of the coefficients such that $$\|f\|_{C^1_b([0,T]\times\RR^6)} \le C.$$ 
\end{prop}

\begin{com} 
	$\;$
	\begin{itemize}
		\itema	If we use a final value condition $f\big|_{t=T}=\Mf$ instead of the initial value condition $f\big|_{t=0}=\Mf$ the problem can be treated completely analogously. The results of Proposition~\ref{CSLVL} and Corollary~\ref{WSLVL} hold true in this case. Only the implicit depiction of a classical solution must be replaced by
		\begin{align}
		\label{EXPLVL2}
		f(t,z) = \Mf\big( Z(T,t,z) \big) - \int\limits_t^T \big[\delx\psi_f\cdot\MC + \Mchi\Phi_{\Ma,f} + \Mb\big]\big(s,Z(s,t,z) \big)  \ds
		\end{align}
		\itemb Suppose that $\MC=0$ and recall that $\Phi_{\Ma,f}$ depends only on $f\big|_{\Bro}$. Hence, if we choose $r_1 = \zeta(r_0)$ then for all $t\in [0,T]$ and $z\in \Bro$,
		\begin{align}
		\label{EXPLVL3}
		f(t,z) = \Mf\big( Z(0,t,z) \big) + \int\limits_0^t \big[\Phi_{\Ma,f} + \Mb\big]\big(s,Z(s,t,z) \big)  \ds
		\end{align}
		because in this case $\Mchi\big(Z(s,t,z)\big) = 1$ as $Z(s,t,\Bro) \subset B_{r_1}(0)$. This means that the values of $f\big\vert_\Bro$ do not depend on the choice of $\Mchi$ as long as \eqref{RCchi} and \eqref{SC2} hold.
	\end{itemize}
\end{com}

\begin{proof}[Proof of Proposition \ref{CSLVL}]
	Let $c>0$ denote a generic constant depending only on $r_0$, $r_2$, $T$ and the norms of the coefficients. For $t\in[0,T]$ and $z\in \RR^6$ let ${Z=(X,V)(s,t,z)}$ denote the solution of the characteristic system with $Z(t,t,z)=z$. Moreover, for $t\in[0,T]$ and $z\in\RR^6$, we define a recursive sequence by $f_0(t,z) := \Mf(z)$ and
	\begin{align*}
	f_{n+1}(t,z) &:= \Mf(Z(0,t,z)) + \int\limits_0^t \big[ \delx\psi_{f_n}\cdot\MC +\Mchi\Phi_{\Ma,f_n} + \Mb \big]\big(s,Z(s,t,z)\big)  \ds.
	\end{align*}
	By induction we can conclude that all $f_n$ are continuous. Then for any fixed $\tau\in[0,T]$ and $n\in\NN$ the functions $\Mf$, $\big[\delx\psi_{f_n}\cdot\MC\big](\tau)$, $\big[\Mchi\Phi_{\Ma,f_n}\big](\tau)$ and $\Mb(\tau)$ are continuous and compactly supported in $B_{r}(0)$ with $r=\max\{r_0,r_2\}$. This directly implies that $f_0(t)$ is compactly supported with $\supp f_0(t)\subset B_{r}(0)$ for all $t\in[0,T]$. Moreover, for any $\tau\in[0,T]$, Lemma \ref{LVLZ} implies that
	\begin{align*}
	\left.
	\begin{aligned}
	\supp \Mf(Z(s,t,\cdot)) &= Z(t,s,\supp \Mf)\\
	\supp \big[\delx\psi_f\cdot\MC](\tau,Z(s,t,\cdot)) &= Z\big(t,s,\supp \delx\psi_{f_n}\cdot\MC(\tau)\big)\\
	\supp \big[\Mchi\Phi_{\Ma,f_n}\big](\tau,Z(s,t,\cdot)) &= Z\big(t,s,\supp\Mchi\Phi_{\Ma,f_n}(\tau)\big)\\
	\supp \Mb\big(\tau,Z(s,t,\cdot)\big) &= Z\big(t,s,\supp \Mb(\tau)\big)
	\end{aligned}
	\right\} \; \subset \; B_{\zeta(r)}(0)\;.
	\end{align*}
	If we choose $\tau=s$ we can inductively deduce that $\supp f_n(t) \subset B_{\zeta(r)}(0)$ for all ${t\in[0,T]}$ and all $n\in\NN$. Finally, by another induction, ${f_n \in C^1(]0,T[\times\RR^6)}$ as the partial derivatives can be recursively described by:\vspace{-2mm}
	\begin{align*}
	&\delt f_0(t,z) = 0, \qquad \delzi f_0(t,z) = \delzi \Mf(z),\\
	&\delt f_{n+1}(t,z) = \delz \Mf(Z(0,t,z))\cdot \delt Z(0) + \delx\psi_{f_n}\cdot\MC(t,z) +\Mchi\Phi_{\Ma,f_n}(t,z) + \Mb(t,z) \\
	&\hspace{22mm}+ \int\limits_0^t \delz\big[ \delx\psi_{f_n}\cdot\MC +\Mchi\Phi_{\Ma,f_n} + \Mb \big]\big(s,Z(s,t,z)\big) \cdot \delt Z(s,t,z) \ds\\
	&\delzi f_{n+1}(t,z) = \delz \Mf(Z(0,t,z))\cdot \delzi Z(0,t,z) \\
	&\hspace{22mm} + \int\limits_0^t \delz\big[ \delx\psi_{f_n}\cdot\MC +\Mchi\Phi_{\Ma,f_n} + \Mb \big]\big(s,Z(s,t,z)\big) \cdot \delzi Z(s,t,z) \ds.
	\end{align*}
	where\vspace{-2mm}
	\begin{align*}
	&\delz\big[ \delx\psi_f\cdot\MC +\Mchi\Phi_{\Ma,f_n} + \Mb \big] \\
	&\quad = \begin{pmatrix} D_x^2\psi_{f_n}\, \MC + D_x\MC\, \delx\psi_{f_n} + \delx\Mchi\,\Phi_{\Ma,f_n} + \Mchi \, \Phi_{\Ma,f_n}' +\delx\Mb \\ D_v\MC\,\delx\psi_{f_n} + \delv\Mchi\,\Phi_{\Ma,f_n} +\delv\Mb \end{pmatrix}.
	\end{align*}
	Using Lemma \ref{LVLZ}, \eqref{ESTPHI3}, \eqref{ESTPHI4} and \cite[Lem.\,2]{knopf}, we obtain the following estimates by a straightforward computation:
	\begin{gather*}
	\|f_{1}(t)-f_0(t)\|_\infty\le c, \|\delt f_{1}(t)- \delt f_0(t)\|_\infty\le c,	\|\delzi f_{1}(t)- \delzi f_0(t)\|_\infty\le c,\\
	\begin{aligned}
	\|f_{n+1}(t)-f_n(t)\|_\infty &\le c \int\limits_0^t \| f_n(s) - f_{n-1}(s) \|_\infty \ds, \\
	\|\delzi f_{n+1}(t)- \delzi f_n(t)\|_\infty &\le c\int\limits_0^t \|f_n(s) - f_{n-1}(s)\|_{W^{1,\infty}}\ds, \\
	\|\delt f_{n+1}(t)- \delt f_n(t)\|_\infty &\le c \int\limits_0^t \Big[ \| \delt f_n(s) - \delt f_{n-1}(s)\|_\infty, \\[-3mm]
	&\qquad\qquad + \| f_n(s) - f_{n-1}(s) \|_{W^{1,\infty}} \Big] \ds.
	\end{aligned}
	\end{gather*}
	Hence there exists some constant $c_*>0$ such that for all $t\in[0,T]$,
	\begin{align*}
	M_{1,0}(t)\le c_* \quad\text{and}\quad M_{n+1,n}(t) \le c_* \int\limits_0^t M_{n,n-1}(s) \ds,\quad n\in\NN
	\end{align*}
	where $M_{m,n}(t)$ denotes the expression
	\begin{align*}
	\max\left\{ \|f_m(t)-f_n(t)\|_\infty,\|\delt f_m(t)-\delt f_n(t)\|_\infty,\|\delz f_m(t)- \delz f_n(t)\|_\infty \right\}
	\end{align*}
	for $m,n\in\NN_0$. Thus by induction,
	\begin{align*}
	M_{n+1,n}(t) \le c_*\frac{t^n}{n!} \le c_*\frac{T^n}{n!}, \quad t\in[0,T],n\in\NN
	\end{align*}
	and hence for $m,n\in\NN$ with $n<m$,
	\begin{align*}
	M_{m,n}(t) \le \sum_{j=n}^{m-1} M_{j+1,j}(t)  \le  \sum_{j=n}^\infty c_*\frac{T^j}{j!} \to 0,\quad n\to\infty\;.
	\end{align*}
	Consequently $(f_n)$ is a Cauchy-sequence in $C_b^1([0,T]\times\RR^6)$ and converges to some function $f\in C_b^1([0,T]\times\RR^6)$ because of completeness. Obviously, as the radius $\zeta(r)$ does not depend on $n$,
	$\supp f(t) \subset  \overline{B_{\zeta(r)}(0)} \subset B_{\zeta(r+1)}(0)$ for all $t\in[0,T]$
	and $f$ satisfies the equation
	\begin{align}
	\label{IMPRF}
	f(t,z) = \Mf(Z(0,t,z)) + \int\limits_0^t \big[ \delx\psi_f\cdot\MC + \Phi_{\Ma,f} +\Mb\big](s,Z(s,t,z))\ds\;.
	\end{align}
	One can easily show that $f$ is a classical solution of \eqref{LVL} by differentiating both sides of \eqref{IMPRF} with respect to $t$. 
	We will finally prove uniqueness by assuming that there exists another solution $\tilde f$ of the initial value problem and define $d:=f-\tilde f$. Then for any $t\in[0,T]$,
	\begin{align*}
	\|d(t)\|_{L^2}^2 = 2\int\limits_0^t\int \delx\psi_{d(s)}\cdot\MC(s)\;d(s) + \Mchi\Phi_{\Ma,d}(s)\; d(s)  \dz\mathrm ds \le c\int\limits_0^t \|d(s)\|_{L^2}^2 \ds
	\end{align*}
	and hence $\|d(t)\|_{L^2}=0$ for all $t\in[0,T]$ by Gronwall's lemma. This directly implies that $f=\tilde f$ which means uniqueness.
\end{proof}

%
%	DEFINITION
%

\begin{defn}
	\label{DWSLVL}
	\hypertarget{HDWSLVL}
	We call $f$ a strong solution of the initial value problem \eqref{LVL} iff the following holds:
	\begin{itemize}
		\item [\textnormal{(i)}] $f\in H^1(]0,T[\times \RR^6) \subset C([0,T];L^2)$.
		\item [\textnormal{(ii)}] $f$ satisfies
		\begin{align*}
		\delt f  + v\cdot\delx f  + \MA\cdot\delv f + (v\times \MB)\cdot\delv f  = \delx\psi_f\cdot\MC + \Phi_{\Ma,f} + \Mb
		\end{align*}
		almost everywhere on $[0,T]\times\RR^6$.
		\item [\textnormal{(iii)}] $f$ satisfies the initial condition $f\big\vert_{t=0} = \Mf$ almost everywhere on $\RR^6$.
		\item [\textnormal{(iv)}] There exists some radius $r>0$ such that $\supp f(t) \subset B_r (0)$, $t\in [0,T]$.
	\end{itemize}
\end{defn}

%
%	COROLLARY
%

\begin{cor}
	\label{WSLVL}
	\hypertarget{HWSLVL}
	We define $r:=\max\{r_0,r_2\}$ and let $C>0$ denote some constant depending only on $r_0,\,r_2$ and the norms of the coefficients.
	\begin{itemize}
		\item [\textnormal{(a)}] Suppose that $\MB  \in L^2(0,T;C^{1,\gamma}(\RR^3;\RR^3)\big)$, $\MC \in L^2\big(0,T;H^1\cap C_b(\RR^6;\RR^3)\big)$, \linebreak $\Mb \in L^2\big(0,T;C_b\cap H^1(\RR^6)\big)$ and $\Mf \in C^1_c(\RR^6)$.
		Moreover, we assume that the regularity conditions \eqref{RCa}, \eqref{RCA}, \eqref{RCchi} and the support conditions \eqref{SC1}, \eqref{SC2} hold.
		Then there exists a unique strong solution $f\in L^\infty\cap H^1(]0,T[\times\RR^6)$ of the initial value problem \eqref{LVL} such that
		\begin{align*}
		\|f\|_{L^\infty(]0,T[\times\RR^6)} + \|f\|_{H^1(]0,T[\times\RR^6)} \le C
		\end{align*}
		and $\supp f(t) \subset B_{\zeta(3+r)}(0)$ for almost all $t\in[0,T]$.
		\item [\textnormal{(b)}] Suppose that $\Mb=0$, $\MC=0$ and $\MB  \in L^2(0,T;C^{1,\gamma}(\RR^3;\RR^3)\big)$.
		Moreover, we assume that the regularity conditions \eqref{RCa}, \eqref{RCf}, \eqref{RCA}, \eqref{RCchi} and the support conditions \eqref{SC1}, \eqref{SC2} hold.
		There exists a unique strong solution \linebreak$f\in W^{1,2}(0,T;C_b) \cap C([0,T];C^1_b)$ of \eqref{LVL} such that
		\begin{align*}
		\|f\|_{L^\infty(]0,T[\times\RR^6)} + \|f\|_{H^1(]0,T[\times\RR^6)} \le C
		\end{align*}
		and $\supp f(t) \subset B_{\zeta(2+r)}(0)$ for almost all $t\in[0,T]$. If $r_1=\zeta(r_0)$, the values of $f\big\vert_{B_{r_0}(0)}$ do not depend on the choice of $\Mchi$ as long as \eqref{RCchi} and \eqref{SC2} hold. 
	\end{itemize}
\end{cor}
\begin{proof}
	To prove (a) we can choose $(\Mb_k) \subset C([0,T];C^1_b)$, $(\MB_k) \subset C([0,T];C^{1,\gamma})$, $(\MC_k)\subset C([0,T];C^1_b)$ and $(\Mf_k)\subset C^2_c(\RR^6)$ such that
	\begin{align*}
	&\Mb_k \to \Mb \text{ in } L^2\big( 0,T;C_b\cap H^1 \big),	&& \|\Mb_k\|_{L^2(0,T;H^1)} \le 2 \|\Mb\|_{L^2(0,T;H^1)}, \\
	& && \|\Mb_k\|_{L^2(0,T;C_b)} \le 2 \|\Mb\|_{L^2(0,T;C_b)}, \\
	&\Mf_k \to \Mf \text{ in } C^1_b(\RR^6),	&& \|\Mf_k\|_{C^1_b} \le 2\|\Mf\|_{C^1_b}\\
	&\MB_k \to \MB \text{ in } L^2\big( 0,T;C^{1,\gamma} \big),	&& \|\MB_k\|_{L^2(0,T;C^{1,\gamma})} \le 2\|\MB\|_{L^2(0,T;C^{1,\gamma})} \\
	&\MC_k \to \MC \text{ in } L^2\big( 0,T;C_b \cap H^1\big) ,	&& \|\MC_k\|_{L^2(0,T;H^1)} \le 2 \|\MC\|_{L^2(0,T;H^1)}, \\
	& && \|\MC_k\|_{L^2(0,T;C_b)} \le 2 \|\MC\|_{L^2(0,T;C_b)} 
	\end{align*}
	and for all $t\in[0,T]$, $\supp \Mb_k(t)$, $\supp \Mf_k$, and $\supp \MC(t) \subset B_{r_0+1}(0)$. Then, due to Proposition \ref{CSLVL}, for every $k\in\NN$ there exists a unique classical solution $f_k$ of \eqref{LVL} to the coefficients $\Ma$, $\Mb_k$, $\Mf_k$, $\MA$, $\MB_k$, $\MC_k$ and $\Mchi$. Moreover for all $t\in[0,T]$,
	\begin{align*}
	\supp f_k(t) \subset B_{\varrho}(0) \quad\text{with}\quad \varrho := \zeta(2 + \max\{r_0,r_2\}) = \zeta(2+r)\;.
	\end{align*}
	Now let $Z_k$ denote the solution of the characteristic system to $\MA$ and $\MB_k$ satisfying $Z_k(t,t,z)=z$ and let $c>0$ denote some generic constant depending only on $T$, $r_0$, $r_2$ and the norms of the coefficients. From Lemma \ref{LVLZ} we know that for any $r>0$ and all $s,t\in[0,T]$,
	\begin{align}
	\label{ESZK} 
	\|Z_k(s,t,\cdot)\|_{L^\infty(B_{r}(0))}< C(r) \tand \|\delz Z_k(s,t,\cdot)\|_{L^\infty(B_r(0))}< C(r)
	\end{align}
	where $C(r)>0$ depends only on $r$, $\|\MA\|_{L^2(0,T;C^1_b)}$ and $\|\MB\|_{L^2(0,T;C^1_b)}$. Then we can conclude from the implicit description \eqref{EXPLVL} that\vspace{-2mm}
	\begin{align*}
	|f_k(t,z)| &\le \|\Mf_k\|_\infty + \int\limits_0^t \|\delx\psi_{f_k}(s)\|_\infty \, \|\MC_k(s)\|_\infty + \|\Phi_{\Ma,f_k}(s)\|_\infty + \|\Mb_k(s)\|_\infty \ds\\
	&\le c + c\int\limits_0^t \|{f_k}(s)\|_\infty  \ds, \quad (t,z)\in[0,T]\times\RR^6.
	\end{align*}
	which yields $\|f_k(t)\|_{L^\infty} \le c$ by Gronwall's lemma. By differentiating \eqref{EXPLVL} and using \eqref{ESZK} the $z$-derivative can be bounded similarly by\vspace{-2mm}
	\begin{align*}
	&\|\delz f_k(t)\|_{L^2}^2 = \|\delz f_k(t)\|_{L^2(B_\varrho(0))}^2 \le c + c\; \int\limits_0^t \|\delz f_k(s)\|_{L^2}^2  \ds
	\end{align*}
	which implies that $\|\delz f_k(t)\|_{L^2}\le c$ for all $t\in[0,T]$. Finally one can easily show that $\|\delt f_k\|_{L^2(0,T;L^2)} \le c$ by expressing $\delt f_k$ by the Vlasov equation.
	Since all $f_k(t)$ are compactly supported in $B_\varrho(0)$ this yields 
	$$\|f_k\|_{L^\infty(]0,T[\times\RR^6)} + \|f_k\|_{H^1(]0,T[\times\RR^6)} \le c\;.$$
	Then, according to the Banach-Alaoglu theorem, there exists $f \in H^1(]0,T[\times\RR^6)$ such that $f_k \rightharpoonup f$ after extraction of a subsequence. Moreover there exists some function ${f^*\in L^\infty(]0,T[\times\RR^6)}$ such that $f_k \overset{*}{\wto} f^*$ up to a subsequence, i.e., a subsequence of $(f_k)$ converges to $f^*$ with respect to the weak-*-topology on $L^1(]0,T[\times\RR^6)^*$. Thus $f=f^*\in L^\infty(]0,T[\times\RR^6)$  $\cap\; H^1(]0,T[\times\RR^6)$. We will now show that $f$ is a strong solution of \eqref{LVL} by verifying the conditions of Definition~\ref{DWSLVL}.\pskip
	
	\textit{Condition} (i) is evident since ${f \in H^1(]0,T[\times\RR^6)}\subset W^{1,2}(0,T;L^2)$ which directly yields ${f\in C([0,T];L^2)}$ by Sobolev's embedding theorem.\pskip
	
	\textit{Condition} (iv) is also obvious because $\supp f_k \subset B_{\varrho}(0)$ for all $k\in\NN$, $t\in[0,T]$. The radius $\varrho$ does not depend on $k$ and satisfies $\varrho <\zeta(3+r)$.\pskip
	
	\textit{Condition} (ii): By Rellich-Kondrachov, $f_k\to f$ in $L^2([0,T]\times\RR^6)$ up to a subsequence. This implies that $\psi_{f_k} \to \psi_f$ and $\Phi_{\Ma,f_k}\to \Phi_{\Ma,f}$ in $L^2([0,T]\times\RR^3)$ and the assertion easily follows.\pskip
	
	\textit{Condition} (iii): Finally, according to Mazur's lemma, there exists some sequence $(\bar f_k)_{k\in\NN} \subset H^1(]0,T[\times\RR^6)$ such that $\bar f_k\to f$ in $H^1(]0,T[\times\RR^6)$ where for all $k\in\NN$, $\bar f_k$ is a convex combination of $f_1,...,f_k$. This means $\bar f_k(0) = \Mf$ and hence
	\begin{align*}
	\|f(0)-\Mf\|_{L^2} \le c\; \|f-\bar f_k\|_{W^{1,2}(0,T;L^2)} \le c\; \|f-\bar f_k\|_{H^1(]0,T[\times\RR^6)} \to 0,\quad k\to \infty\;.
	\end{align*}
	
	Consequently $f$ is a strong solution but we still have to prove uniqueness. We assume that there exists another strong solution $\tilde f$ and define $d:=f-\tilde f$. Then, by the fundamental theorem of calculus,
	\begin{align*}
	\|d(t)\|_{L^2}^2 
	= 2\int\limits_0^t\int \delx\psi_{d(s)}\cdot\MC(s)\;d(s) + \Mchi\Phi_{\Ma,d}(s)\; d(s)  \dz\mathrm ds
	\le c\int\limits_0^t \|d(s)\|_{L^2}^2 \ds
	\end{align*}
	for all $t\in[0,T]$. Hence $\|f(t)-\tilde f(t)\|_{L^2}^2 = \|d(t)\|_{L^2}^2 = 0$ for every $t\in[0,T]$ by Gronwall's lemma. This proves (a).\pskip
	
	To prove (b) we only have to approximate $\MB$. Therefore we choose some sequence $(\MB_k) \subset C([0,T];C^{1,\gamma})$ such that
	\begin{align*}
	\|\MB_k - \MB\|_{L^2(0,T;C^{1,\gamma})}\to 0,  \tand \|\MB_k\|_{L^2(0,T;C^{1,\gamma})} \le 2\|\MB\|_{L^2(0,T;C^{1,\gamma})}, \;\; k\in\NN.
	\end{align*}
	Then for any $k\in\NN$ there exists a unique classical solution $f_k$ of the system \eqref{LVL} to the coefficients $\Ma$, $\Mf$, $\MA$, $\MB_k$ and $\Mchi$ according to Proposition \ref{CSLVL}. Recall that for all $t\in [0,T]$, $\supp f_k(t) \subset B_{\varrho}(0)$ where $\varrho := \zeta(r+1)$ with $r = \max\{r_0,r_2\}$. Again, let $Z_k$ denote the solution of the characteristic system to $\MA$ and $\MB_k$ satisfying $Z_k(t,t,z)=z$ and in the following the letter $c$ denotes some generic positive constant depending only on $T$, $r_0$, $r_2$ and the norms of the coefficients. Now for all $s,t\in[0,T]$ (where $s\le t$ without loss of generality) and $z\in B_{\varrho}(0)$,$\color{white}{\Big|}$\vspace{-2mm}
	\begin{align*}
	|Z_k(s)-Z_j(s)| 
	&\le \int\limits_s^t c\;(1+\|D_x \MA(\tau)\|_\infty+\|D_x \MB_k(\tau)\|_\infty)\; |Z_k(\tau) - Z_j(\tau)| \dtau \\
	&\qquad+ c\int\limits_0^T \|\MB_k(\tau) - \MB_j(\tau)\|_\infty \dtau
	\end{align*}
	which implies that $\|Z_k(s,t,\cdot)-Z_j(s,t,\cdot)\|_{L^\infty(B_\rho(0))} \le c\;\|\MB_k-\MB_j\|_{L^2(0,T;L^\infty)} $. Similarly, for any $i\in\{ 1,...,6 \}$ the difference of the $i$-th derivative can be bounded by
	\begin{align*}
	&|\delzi Z_k(s)- \delzi Z_j(s)|\\
	&\le c\int\limits_s^t \Big[(1+\|\MA(\tau)\|_{C^{1,\gamma}}+\|\MB_k(\tau)\|_{C^{1,\gamma}})\;|\delzi Z_k(\tau)- \delzi Z_j(\tau)|  \\[-3mm]
	&\qquad\qquad + (1+\|\MA(\tau)\|_{C^{1,\gamma}}+\|\MB_k(\tau)\|_{C^{1,\gamma}})\;\|Z_k(\tau)- Z_j(\tau)\|_\infty^\gamma \\
	&\qquad\qquad  + \|\MB_k(\tau)-\MB_j(\tau)\|_{C^1_b} \,\Big] \dtau\\
	&\le \int\limits_s^t c\;(1+\|\MA(\tau)\|_{C^{1,\gamma}}+\|\MB_k(\tau)\|_{C^{1,\gamma}})\;|\delzi Z_k(\tau)- \delzi Z_j(\tau)| \dtau\\
	&\qquad  +c\;  \|\MB_k-\MB_j\|_{L^2(0,T;C^1_b)}^\gamma
	\end{align*}
	for all $s,t\in[0,T]$ and $z\in B_{\varrho}(0)$. Thus
	\begin{align*}
	\|\delz Z_k(s)- \delz Z_j(s)\|_{L^\infty(B_\varrho(0))} &\le  c\; \|\MB_k-\MB_j\|_{L^2(0,T;C^1_b)}^\gamma\,.
	\end{align*}
	Now for all $t\in[0,T]$, $z\in B_{\varrho}(0)$,
	\begin{align*}
	|f_k(t,z)-f_j(t,z)| &\le \|D\Mf\|_\infty \; |Z_k(0,t,z)-Z_j(0,t,z)| + c \int\limits_0^t \|f_k(\tau)-f_j(\tau)\|_\infty \mathrm d\tau \\
	&\le c\;\|\MB_k-\MB_j\|_{L^2(0,T;L^\infty)} + c \int\limits_0^t \|f_k(\tau)-f_j(\tau)\|_\infty \dtau
	\end{align*}
	and thus Gronwall's lemma implies that
	$$\|f_k-f_j\|_{L^\infty(0,T;L^\infty)} \le c\;\|\MB_k-\MB_j\|_{L^2(0,T;L^\infty)} $$
	Similarly, for all $t\in[0,T],z\in\Br$,
	\begin{align*}
	|\delz f_k(t,z)- \delz f_j(t,z)| \le c\|\MB_k-\MB_j\|_{L^2(0,T;C^1_b)}^\gamma + c \int\limits_0^t \|\delz f_k(\tau)- \delz f_j(\tau)\|_\infty \dtau
	\end{align*}
	and consequently $\|\delz f_k- \delz f_j\|_{L^\infty(0,T;L^\infty)} \le c\;\|\MB_k-\MB_j\|_{L^2(0,T;C^{1,\gamma})}^\gamma$. By expressing $\delt f_k$ and $\delt f_j$ by their corresponding Vlasov equation we can easily compute the estimate $\|\delt f_k- \delt f_j\|_{L^2(0,T;C_b)} \le c\;\|\MB_k-\MB_j\|_{L^2(0,T;C^{1,\gamma})}^\gamma$. \pskip
	
	This means that $(f_k)$ is a Cauchy sequence in $W^{1,2}(0,T;C_b)\cap C([0,T];C^1_b)$ and thus it converges to some function $f\in W^{1,2}(0,T;C_b)\cap C([0,T];C^1_b)$ because of completeness. Note that for all $t\in[0,T]$, $\supp f(t) \subset B_{\zeta(r+2)}$. From the strong convergence one can easily conclude that $f$ satisfies the system \eqref{LVL} almost everywhere and thus $f$ is a strong solution according to Definition \ref{DWSLVL}. \pskip
	
	Moreover, by the definition of convergence, we can find $k\in\NN$ such that 
	$\|f-f_k\|_{W^{1,2}(0,T;C_b)} + \|f-f_k\|_{C(0,T;C^1_b)} \le 1$
	and consequently
	\begin{align*}
	&\|f\|_{W^{1,2}(0,T;C_b)} + \|f\|_{C(0,T;C^1_b)} \le 1 + \|f_k\|_{W^{1,2}(0,T;C_b)} + \|f_k\|_{C(0,T;C^1_b)} \le c
	\end{align*}
	as the sequence $(f_k)$ is bounded in $C^1_b(]0,T[\times\RR^6)$ according to Proposition \ref{CSLVL} and $\MB_k$ is bounded by $\|\MB_k\|_{L^2(0,T;C^{1,\gamma})} \le 2\|\MB\|_{L^2(0,T;C^{1,\gamma})}$. \bpskip
	
	We will now assume that $r_1 = \zeta(r_0)$. As it has already been discussed in the comment to Proposition \ref{CSLVL} the values of $f_k\vert_\Bro$ do not depend on the choice of $\Mchi$ as long as \eqref{RCchi} and \eqref{SC2} hold. As $f_k\vert_\Bro$ converges to $f\vert_\Bro$ uniformely on $[0,T]\times\Bro$ this result holds true for $f\vert_\Bro$.
\end{proof}

\section{Fr\'echet differentiability of the field-state operator}

Again, let $K>0$ be arbitrary. We can now use the results of Section 5.1 to establish Fr\'echet differentiability of the control state operator on $\IBB$ (that is the interior of $\BB$).

\begin{thm}
	\label{FDCSO}
	Let $f.$ be the field-state operator as defined in \textnormal{\cite[Def.\,13]{knopf}}. For all \linebreak$B\in\BB$, $H\in \VV$ there exists a unique strong solution $f_B^H\in L^\infty\cap H^1(]0,T[\times\RR)$ $\subset C([0,T];L^2)$ of the initial value problem
	\begin{displaymath}
	\refstepcounter{equation}
	\label{FDEQ}
	\left\{
	\begin{aligned}
	&\delt f + v\cdot\delx f - \delx\psi_{f_B}\hspace{-1pt}\cdot\delv f - \delx\psi_f\hspace{-1pt}\cdot\delv f_B
	+ (v\stimes B)\hspace{-1pt}\cdot\delv f + (v\stimes H)\hspace{-1pt}\cdot\delv f_B = 0 \\[0.25cm]
	&f\big|_{t=0}=0 \hspace{272pt}\textnormal{(\theequation)}
	\end{aligned}
	\right.	
	\end{displaymath}
	with $\supp f(t) \subset B_\varrho(0)$ for all $t\in [0,T]$ and some radius $\varrho>0$ depending only on $T,K,\mathring f$ and $\beta$. Then the following holds:
	\begin{itemize}
		\item [\textnormal{(a)}] The field-state operator $f.$ is Fr\'echet differentiable on $\IBB$ with respect to the $C([0,T];L^2(\RR^6))$-norm, i.e., for any $B\in\IBB$ there exists a unique linear operator $f'_B: \VV\to C([0,T];L^2(\RR^6))$ such that
		\begin{gather*}
		\forall \eps>0\; \exists \delta > 0\; \forall H\in \VV \text{ with } \|H\|_{\VV} < \delta :\\[2mm]
		B+H \in \IBB\tand \frac{\| f_{B+H} - f_B - f'_B[H] \|_{C([0,T];L^2)}}{\|H\|_{\VV}} < \eps\,.
		\end{gather*}
		The Fr\'echet derivative is given by $f'_B[H]=f_B^H$ for all $H\in\VV$.
		\item [\textnormal{(b)}] For all $B,H\in\IBB$, the solution $f_B^H$ depends Hölder-continuously on $B$ in such a way that there exists some constant $C>0$ depending only on $\mathring f, T, K$ and $\beta$ such that for all $A,B\in\IBB$,
		\begin{align}
		\label{CFF}
		\underset{\|H\|_\VV \le 1}{\sup}\;\|f_A'[H] - f_B'[H]\|_{L^2(0,T;L^2)} \le C \;\|A-B\|_{\VV}^{\gamma}.
		\end{align}
	\end{itemize}
\end{thm}

\begin{com}
	As $K>0$ was arbitrary the obove results hold true on $\IBBB$ instead of $\IBB$. Hence they are especially true for $B\in\BB$. 
\end{com}

\begin{proof} Let $C$ denote some generic positive constant depending only on $\mathring f$, $K$, $T$ and $\beta$. First note that the system \eqref{FDEQ} is of the type \eqref{LVL} since the coefficients of \eqref{FDEQ} satisfy the regularity and support conditions of Corollary~\ref{WSLVL}. Hence \eqref{FDEQ} has a strong solution $f_B^H\in L^\infty\cap H^1(]0,T[\times\RR^6)$. To prove Fr\'echet differentiability of the field-state operator we must consider the difference $f_{B+H}-f_B$ with $B\in\IBB$ and $H\in\VV$ such that $B+H\in\IBB$. Therefore we will assume that
	$\|H\|_\VV<\delta$ for some sufficiently small $\delta>0$.
	Now we expand the nonlinear terms in the Vlasov equation \eqref{VP} to pick out the linear parts. We have
	\begin{align*}
	&\delx\psi_{f_{B+H}}\cdot\delv f_{B+H} - \delx\psi_{f_{B}}\cdot\delv f_{B}\\
	&\quad = \delx\psi_{f_B}\cdot \delv(f_{B+H}-f_B) + \delx\psi_{(f_{B+H}-f_B)}\cdot\delv f_B + \mathcal R_1,\\[2mm]
	&\big(v\times (B+H)\big)\cdot\delv f_{B+H} - (v\times B)\cdot\delv f_{B} \\
	&\quad = (v\times B)\cdot\delv(f_{B+H}-f_B) + (v\times H)\cdot\delv f_B + \mathcal R_2
	\end{align*}
	where $\mathcal R_1 := \delx\psi_{(f_{B+H}-f_B)}\cdot\delv(f_{B+H}-f_B)$ and $\mathcal R_2 := (v\times H)\cdot\delv(f_{B+H}-f_B)$ are nonlinear remainders. Then $\mathcal R:=\mathcal R_1-\mathcal R_2$ lies in $L^2(0,T;H^1\cap C_b)$ and from \cite[Lem.\,2]{knopf} and \cite[Cor.\,14]{knopf} we can conclude that $\|\mathcal R\|_{L^2(0,T;L^2)} \le C \|H\|_{\VV}^{1 + \gamma}$. Obviously $f_{B+H}-f_B$ solves the initial value problem
	\begin{displaymath}	
	\refstepcounter{equation}
	\label{EQFR}
	\hspace{-2pt}\left\{
	\begin{aligned}
	&\hspace{-2pt}\delt f + v\cdot\delx f - \delx\psi_{f_B}\hspace{-1pt}\cdot\delv f - \delx\psi_f\hspace{-1pt}\cdot\delv f_B
	+ (v\stimes B)\hspace{-1pt}\cdot\delv f + (v\stimes H)\hspace{-1pt}\cdot\delv f_B \hspace{-1pt}= \mathcal R\\[0.25cm]
	&\hspace{-2pt}f\big|_{t=0}=0 \hspace{277pt}\textnormal{(\theequation)}
	\end{aligned}
	\hspace{-2pt}\right.
	\end{displaymath}
	almost everywhere on $[0,T]\times\RR^6$. From Corollary \ref{WSLVL}\,(a) we know that this solution is unique. Also according to Corollary \ref{WSLVL}\,(a) the system
	\begin{align}
	\label{EQFR2}
	\begin{cases}
	\delt f + v\cdot\delx f - \delx\psi_{f_B}\cdot\delv f - \delx\psi_f\cdot\delv f_B + (v\times B)\cdot\delv f = \mathcal R\;,\\[0.25cm]
	f\big|_{t=0}=0\;.
	\end{cases}
	\end{align}
	has a unique strong solution $f_{\mathcal R}$. Then $f_B^H+f_{\mathcal R}$ is a solution of \eqref{EQFR} due to linearity and thus $f_{B+H}-f_B = f_B^H + f_{\mathcal R}$ because of uniqueness. It holds that
	\begin{align*}
	\|f_{\mathcal R}(t)\|_{L^2}^2 &= 2 \int\limits_0^t \int f_{\mathcal R}(s)\; \delt f_{\mathcal R}(s) \dz\mathrm ds = 2 \int\limits_0^t \int f_{\mathcal R}\; \big( \delx\psi_{f_{\mathcal R}}\cdot\delv f_B  + \mathcal R \big) \dz\mathrm ds\\
	& \le C \int\limits_0^t   \|f_{\mathcal R}(s)\|_{L^2}^2 + \|f_{\mathcal R}(s)\|_{L^2}\;\|\mathcal R(s)\|_{L^2} \ds\;.
	\end{align*}
	Applying first the standard version and then the quadratic version of Gronwall's lemma (cf. Dragomir \cite[p.\,4]{dragomir}) yields
	\begin{align*}
	\|f_{\mathcal R}(t)\|_{L^2} \le C\;\|\mathcal R\|_{L^2(0,T;L^2)} \le C\;\|H\|_{\VV}^{1+\gamma}
	\end{align*}
	Let now $\eps>0$ be arbitrary. Then for all $t\in[0,T]$,
	\begin{align*}
	\frac{\| f_{B+H} - f_B - f_B^H \|_{C([0,T];L^2)}}{\|H\|_{\VV}} = \frac{\| f_{\mathcal R} \|_{C([0,T];L^2)}}{\|H\|_{\VV}}
	\le C\;\|H\|_{\VV}^{\gamma} < \eps
	\end{align*}
	if $\delta$ is sufficiently small. Hence assertion (a) is proved and the Fr\'echet derivative is determined by the system \eqref{FDEQ}. \pskip
	
	To prove (b) suppose that $A,B\in\IBB$ and $H\in\VV$ with $\|H\|_\VV\le 1$. Now, we choose sequences $(A_k),(B_k),(H_k)\subset C([0,T];W^{2,\beta})\subset C([0,T];C^{1,\gamma})$ such that $A_k\to A$, $B_k\to B$, $H_k\to H$ in $L^2(0,T;C^{1,\gamma})$ if $k$ tends to infinity. From Corollary \ref{WSLVL} (and its proof) we can conclude that
	\begin{gather*}
	\|f_{A_k}^{H_k}\|_{H^1(]0,T[\times\RR^6)} \le C \tand \|f_{B_k}^{H_k}\|_{H^1(]0,T[\times\RR^6)}  \le C,\\
	f_{A_k}^{H_k} \rightharpoonup f_A^H \tand f_{B_k}^{H_k} \rightharpoonup f_B^H \quad\text{in}\; H^1(]0,T[\times\RR^6)\,.
	\end{gather*}
	Since the $(x,v)$-supports of all occurring functions are contained in some ball $B_\varrho(0)$ whose radius $r$ depends only on $\mathring f$, $K$, $T$ and $\beta$ but not on $k$, we can apply the Rellich-Kondrachov theorem to obtain\vspace{-0.2cm}
	\begin{align*}
	f_{A_k}^{H_k} \to f_A^H \tand f_{B_k}^{H_k} \to f_B^H \quad\text{in}\; L^2([0,T]\times\RR^6)
	\end{align*}
	up to a subsequence. As $A_k,B_k$ and $H_k$ satisfy the regularity condition \eqref{RCB}, $f_{A_k}^{H_k}$ and $f_{B_k}^{H_k}$ are classical solutions and can be described implicitely by the representation formula \eqref{EXPLVL}. Note that \cite[Lem.\,8]{knopf} holds true for $\varrho$ instead of $R$. Hence for all $s,t\in[0,T]$,
	\begin{gather*}
	\|Z_{F}(s,t,\cdot)\|_{L^\infty(B_\varrho(0))} \le C, \quad \|\delz f_F(s)\|_\infty \le C, \quad \|D_z^2 f_F\|_{L^2(0,T;L^2)} \le C 
	\end{gather*}
	for all $F\in\big\{A_k,B_k \,\big\vert\, k\in\NN \big\}$. Also recall that we know from \cite[Lem.\,9]{knopf} (with $\varrho$ instead of $R$) that for all $s,t\in[0,T]$,
	\begin{gather*}
	\|f_{A_k}(s) - f_{B_k}(s)\|_{\infty} \le C\; \|A_k - B_k\|_{\VV}\;,\\[0.5mm]
	\|\delz f_{A_k}(s) - \delz f_{B_k}(s)\|_{\infty} \le C\; \|A_k - B_k\|_{\VV}^{\gamma}\;,\\
	\|Z_{A_k}(s,t,\cdot) - Z_{B_k}(s,t,\cdot)\|_{L^\infty(B_\varrho(0))} \le C\; \|A_k - B_k\|_{\VV}\;.
	\end{gather*}
	Then we can conclude from the implicit description \eqref{EXPLVL} that
	\begin{align*}
	&\|f_{A_k}^{H_k}(t) - f_{B_k}^{H_k}(t)\|_{L^2}\\
	&\le C\int\limits_0^t  \big( 1 + \|f_{A_k}(s)\|_{H^2} + \|H_k(s)\|_{W^{2,\beta}}\big)\, \|Z_{A_k}(s) - Z_{B_k}(s)\|_{L^\infty(B_\varrho(0))} \ds \\
	&\quad + C\int\limits_0^t \big( 1+ \|H_k(s)\|_{W^{2,\beta}} \big)\,\|\delv f_{A_k}(s) - \delv f_{B_k}(s) \|_{L^\infty} \ds \\
	&\quad + C\int\limits_0^t  \|f_{A_k}^{H_k}(s) - f_{B_k}^{H_k}(s)\|_{L^2} \ds
	\end{align*}
	and thus by Gronwall's lemma,\vspace{-1mm}
	\begin{align*}
	\|f_{A_k}^{H_k} - f_{B_k}^{H_k}\|_{L^2(0,T;L^2)} \le C\; \|f_{A_k}^{H_k} - f_{B_k}^{H_k}\|_{L^\infty(0,T;L^2)} \le C\; \|A_k - B_k\|_{\VV}^{\gamma}.
	\end{align*}
	If $k\to\infty$, we obtain $\|f_{A}^{H} - f_{B}^{H}\|_{L^2(0,T;L^2)} \le C\; \|A - B\|_{\VV}^{\gamma}$ that is (b). \end{proof}

%%
%%	OPTIMAL CONTROL WITH B-FIELDS
%%

\section{An optimal control problem with a tracking type cost functional}

%\subsection{The model and existence of a globally optimal solution}

We will now consider the model problem that was introduced in \cite[Sect.\,5]{knopf}. Let $\mathring f \in C^2_c(\RR^6)$ be any given initial datum and let $T>0$ be some fixed final time. Now, we want to find a control $B\in\BB$ such that the distribution function $f_B$ at time $T$ matches a desired distribution function $f_d\in C^2_c(\RR^6)$ as closely as possible. This is to be achieved by minimizing the $L^2$-difference $\|f_B(T)-f_d\|_{L^2}$. Therefore our optimization problem reads as follows:
\begin{align}
\label{OP1}
\begin{aligned}
&\text{Minimize} &&J(B) = \frac 1 2 \|f_B(T)-f_d\|_{L^2(\RR^6)}^2 + \frac \lambda 2  \|D_x B\|_{L^2([0,T]\times\RR^3;\RR^{3\times 3})}^2, \\
&\text{s.t.} && B\in\BB.
\end{aligned}
\end{align}
Here $\lambda$ is a nonnegative parameter and the field $B$ is the control in our model. Since $\|f(t)\|_{p} = \|\mathring f\|_{p}$ for all $1\le p\le \infty$, $t\in[0,T]$ it makes sense to assume that $\|f_d\|_{p} = \|\mathring f\|_{p}$ for all $1\le p\le \infty$ because otherwise the exact matching $f(T)=f_d$ would be impossible from the beginning. \pskip

In \cite[Thm.\,16]{knopf} we have already proved that this optimal control problem has at least one optimal solution. Since the control-state operator $f.$ is nonlinear we cannot expect $J$ to be convex. Of course the regularization term is strictly convex with respect to $B$ if $\lambda>0$ but if $\lambda$ is rather small (which makes sense in this model) there is no chance that this property can be transferred to $J$. Hence we can not conclude that there is only one globally optimal solution. Of course the optimal control problem may also have several locally optimal solutions. In the following subsection, these locally optimal solutions will be characterized by necessary optimality conditions of first order.

\subsection{Necessary conditions for local optimality}

A locally optimal solution is defined as follows:\vspace{-0.2cm}

\begin{defn}
	A control $\B*\in\BB$ is called a locally optimal solution of the optimization problem \eqref{OP1} iff there exists $\delta>0$ such that
	\begin{align*}
	J(\B*)\le J(B) \quad\text{for all}\quad B\in B_\delta(\B*) \cap \BB
	\end{align*}
	where $B_\delta(\B*)$ is the open ball in $\VV$ with radius $\delta$ and center~$\B*$.\pskip
\end{defn}

To establish necessary optimality conditions of first order we need Fréchet differentiability of the cost functional $J$.

\begin{lem}
	\label{FDJ-1}
	The cost functional $J$ is Fréchet differentiable on $\BB$ with Fréchet derivative
	\begin{align*}
	J'(B)[H] = \langle f_B(T)-f_d , f_B'(T)[H] \rangle_{L^2(\RR^6)} + \lambda \langle D_x B, D_x H \rangle_{L^2([0,T]\times\RR^3;\RR^{3\times 3})}
	\end{align*}
	for all $H\in \VV$. Let $\B*\in\BB$ be a locally optimal solution of the optimization problem \eqref{OP1}. Then
	\begin{align*}
	J'(\B*)[H] \begin{cases} =0, &\text{if}\; \B*\in\IBB \\ \ge 0, &\text{if}\; \B*\in\partial\BB \end{cases}, \quad H\in\BB \text{ with } \B*+H \in \BB.
	\end{align*}
\end{lem}

\begin{proof} As the control-state operator is Fréchet differentiable on $\BB$ so is the cost functional $J$ by chain rule. Thus, the function $[0,1]\ni t\mapsto J(\B*+tH)\in\RR$ is differentiable with respect to $t$ and since $\B*$ is also a local minimizer of this function, we have
	\begin{align*}
	0 &\le \ddt J(\B* + tH)\big\vert_{t=0} = \left(J'(\B*+tH)\left[\ddt(\B*+tH)\right]\right)\Big\vert_{t=0} = J'(\B*)[H]
	\end{align*}
	for any $H\in\BB$ with $B+H\in\BB$. If $\B*$ is an inner point of $\BB$ this line even holds with "$=$" instead of "$\le$".
\end{proof}

If we consider $\BB$ as a subset of $L^2([0,T]\times \RR^3;\RR^3)$ it might be possible to find an adjoint operator $\big(f_B'(T)\big)^*\hspace{-3pt}: C([0,T];L^2) \to L^2([0,T]\times \RR^3;\RR^3)$ of $f_B'(T)$. Then, by integration by parts,\vspace{-2mm}
\begin{align*}
J'(B)[H] &= \langle f_B(T)-f_d , f_B'(T)[H] \rangle_{L^2(\RR^6)} + \lambda \sum_{i=1}^3 \langle \delxi B, \delxi H \rangle_{L^2([0,T]\times\RR^3;\RR^3)} \\
& = \langle \big(f_B'(T)\big)^*[f_B(T)-f_d]  - \lambda\;  \laplace_x B,H \rangle_{L^2([0,T]\times\RR^3;\RR^3)}
\end{align*}
for all $H\in\VV$. This means that the derivative $J'$ would have the explicit description $J'(B) = \big(f_B'(T)\big)^*[f_B(T)-f_d]  - \lambda\; \laplace_x B$. If now $\B* \in \interior \BB$ were a locally optimal solution it would satisfy the semilinear Poisson equation
\begin{align*}
\label{ADJOP}
-\laplace_x B = -\frac 1 \lambda \big(f_B'(T)\big)^*[f_B(T)-f_d] \;.
\end{align*}
In general such an adjoint operator is not uniquely determined. This means that we cannot deduce uniqueness of our optimal solution. A common technique to find an adjoint operator is the \textbf{Lagrangian technique}. For $B\in\BB$ and $f,g\in H^1(]0,T[\times\RR^6)$ with $\supp f(t)\subset\BR$ for all $t\in[0,T]$ we define
\begin{align*}
\Lag(f,B,g) &:= \frac 1 2 \|f(T)-f_d\|_{L^2} + \frac \lambda 2  \|D_x B\|_{L^2}^2 \\
&\quad - \int\limits_{[0,T]\times \RR^6} \big( \delt f +v\cdot \delx f - \delx\psi_f\cdot\delv f + (v\times B)\cdot\delv f \big)\; g \dtxv\;.
\end{align*}
$\Lag$ is called the \textbf{Lagrangian}. Obviously, by integration by parts,
\begin{align*}
\Lag(f,B,g)
&= \frac 1 2 \|f(T)-f_d\|_{L^2} + \frac \lambda 2 \|D_x B\|_{L^2}^2 + \langle g(0),f(0) \rangle_{L^2} - \langle g(T),f(T) \rangle_{L^2}\\
&\quad + \int\limits_{[0,T]\times \RR^6} \big( \delt g +v\cdot \delx g - \delx\psi_f\cdot\delv g + (v\times B)\cdot\delv g \big)\; f \dtxv\;.
\end{align*}
In the definition of the Lagrangian $f$, $B$ and $g$ are independent functions. However, inserting $f=f_B$ yields
\begin{align}
J(B) = \Lag(f_B,B,g), \quad B\in\BB,\; g\in H^1(]0,T[\times\RR^6)\;.
\end{align}
It is important that this equality does not depend on the choice of $g$. Since $\Lag$ is Fréchet differentiable with respect to $f$ in the $H^1(]0,T[\times\RR^6)$-sense and with respect to $B$ in the $\VV$-sense we can use this fact to compute the derivative of $J$ alternatively. By chain rule,
\begin{align}
\label{DJLAG}
J'(B)[H] = \big(\partial_f \Lag\big)(f_B,B,g)\big[f_B'[H]\big] + \big(\partial_B \Lag\big)(f_B,B,g)[H]
\end{align}
for all $B\in\BB$, $H\in\VV$ and any $g\in H^1(]0,T[\times\RR^6)$. Here $\partial_f \Lag$ and $\partial_B \Lag$ denote the partial Fréchet derivative of $\Lag$ with respect to $f$ and $B$. We will now fix $f,g$ and $B$. Then
\begin{align}
\label{DFLAG}
(\partial_f \Lag)(f,B,g)[h]
%
%	&= \langle f(T)-f_d, h(T) \rangle_{L^2} - \langle g(T) , h(T) \rangle_{L^2} + \langle g(0), h(0) \rangle_{L^2} \notag\\
%	&\quad + \int\limits_{[0,T]\times \RR^6} \big( \delt g +v\cdot\delx g - \delx\psi_f\cdot\delv g + (v\times B)\cdot\delv g \big)\; h \dtxv \notag\\
%	&\quad + \int\limits_{[0,T]\times \RR^6} \delx\psi_h\cdot \delv f\; g \dtxv \notag\\[0.25cm]
%
&= \langle f(T)-f_d, h(T) \rangle_{L^2} - \langle g(T) , h(T) \rangle_{L^2} + \langle g(0), h(0) \rangle_{L^2} \notag\\
&\; + \hspace{-4pt} \int\limits_{[0,T]\times \RR^6} \hspace{-8pt} \big( \delt g +v\cdot \delx g - \delx\psi_f\cdot\delv g + (v\times B)\cdot\delv g \big)\; h \dtxv \notag\\
&\; - \hspace{-4pt} \int\limits_{[0,T]\times \RR^6} \Phi_{f,g}(t,x,v)\; h \dtxv
\end{align}
for all $h\in H^1(]0,T[\times\RR^6)$ with $\supp h(t) \subset\BR,\, t\in[0,T]$ where $\Phi_{f,g}$ is given by \eqref{DEFPHI}. Moreover,
\begin{align}
\label{DBLAG}
\begin{aligned}
&(\partial_B \Lag)(f,B,g)[H] = \lambda \langle D_x B, D_x H \rangle_{L^2} \;- \hspace{-4pt}\int\limits_{[0,T]\times \RR^6} (v\times H)\cdot\delv f\; g \dtxv \\
& \quad = \hspace{-4pt} \int\limits_{[0,T]\times \RR^3} \left[ - \lambda \laplace_x B +  \int\limits_{\RR^3} v\times \delv f\; g \dv \right] \cdot H \dtx
\end{aligned} 
\end{align}
for all $H\in\VV$. Apparently, the derivative with respect to $B$ looks pretty nice while the derivative with respect to $f$ is rather complicated. However if we insert those terms in \eqref{DJLAG} we can still choose $g$. Now the idea of the Lagrangian technique is to choose $g$ in such a way that the term $(\partial_f \Lag)(f_B,B,g)[f_B'[H]]$ vanishes. \pskip

We consider the following final value problem which we will call the \textbf{costate equation}
\begin{align}\vspace{-2mm}
\label{COSTEQ}
\begin{cases}
\delt g + v\cdot\delx g  - \delx\psi_{f_B}\cdot\delv g + (v\times B)\cdot\delv g = \Phi_{f_B,g}\,\chi \\[0.15cm]
g\big\vert_{t=T}=f_B(T)-f_d
\end{cases}
\end{align}
where $\chi \in C^2_c(\RR^6;[0,1])$ with $\chi = 1$ on $B_{R_Z}(0)$ and $\supp \chi \in B_{2R_Z}(0)$ denotes an arbitrary but fixed cut-off function. Here $R_Z$ is the constant from \cite[Lem.\,8]{knopf}, i.e., for all $s,t\in[0,T]$, $Z_B(s,t,\BR) \subset \BRZ$. Existence and uniqueness of a strong solution to this system will be established in the following theorem:

%
%	THEOREM
%

\begin{thm}
	\label{ADJS}
	\hypertarget{HADJS}
	Let $B\in\BB$ be arbitrary. The costate equation \eqref{COSTEQ} has a unique strong solution
	$g_B\in W^{1,2}\big(0,T;C_b(\RR^6)\big)\cap C\big([0,T];C^1_b(\RR^6)\big) \cap L^\infty\big(0,T;H^2(\RR^6)\big)$ with compact support $\supp g_B(t) \subset B_{R^*}(0)$ for all ${t\in[0,T]}$ and some radius $R^*>0$ depending only on $\mathring f,f_d,T,K$ and $\beta$. \pskip
	
	In this case $g_B\big\vert_{\BR}$ does not depend on the choice of $\chi$. \pskip
	
	Moreover $g_B$ depends Hölder-continuously on $B$ in such a way that there exists some constant $C\ge 0$ depending only on $\mathring f,f_d,T,K,\beta$ and $\|\chi\|_{C^1_b}$ such that for all $B,H\in\BB$,
	\begin{align}
	\label{HCG}
	\|g_B - g_H\|_{W^{1,2}(0,T;C_b)} + \|g_B - g_H\|_{C([0,T];C^1_b)} &\le C\|B-H\|_{\VV}^{\gamma}.
	\end{align}
\end{thm}

\begin{com}
	Note that only the values of $g_B$ on the ball $\BR$ will matter in the following approach. Therefore it is essential that those values are not influenced by the cut-off function $\Mchi$.
\end{com}

\begin{proof}
	$\,$\textit{Step 1}: Obviously the system \eqref{COSTEQ} has a unique strong solution $g_B$ in the sense of Corollary~\ref{WSLVL}\,(a). Unfortunately the coefficients do not satisfy the stronger regularity conditions of Corollary~\ref{WSLVL}\,(b) as the final value $f_B(T)-f_d$ is not in $C^2_c(\RR^6)$. However, because of linearity, it holds that $g_B = \tg_B - h_B$ where $\tg_B$ is a solution of
	\begin{align*}
	\delt \tg + v\cdot\delx \tg  - \delx\psi_{f_B}\cdot\delv \tg + (v\times B)\cdot\delv \tg = \Phi_{f_B,\tg}\ \chi, \qquad 
	\tg\big\vert_{t=T}=f(T)
	\end{align*}
	and $h_B$ is a solution of
	\begin{align*}
	\delt h + v\cdot\delx h  - \delx\psi_{f_B}\cdot\delv h + (v\times B)\cdot\delv h = \Phi_{f_B,h}\ \chi, \qquad 
	h\big\vert_{t=T}=f_d
	\end{align*}
	Now the first system has a unique strong solution in the sense of Corollary~\ref{WSLVL}\,(a) and the second one possesses a strong solution in the sense of Corollary~\ref{WSLVL}\,(b) since ${f_d\in C^2_c(\RR^6)}$. Indeed the solution $\tg_B$ is much more regular. As $\Phi_{f_B,f_B}=0$ one can easily see that $f_B$ is a solution of the first system and thus, because of uniqueness, $\tg_B=f_B$. Consequently $g_B = f_B - h_B$ lies in the space $W^{1,2}(0,T;C_b)\cap C([0,T];C^1_b)$. Due to Corollary~\ref{WSLVL}\,(b) the values of $h_B$ on $B_R(0)$ do not depend on the choice of $\chi$. Of course $f_B$ does not depend on $\chi$ either and hence $g_B\big\vert_{\BR}$ does not depend on the choice of $\chi$. \pskip
	
	\textit{Step 2}: We will now prove the Hölder estimate. It suffices to establish the result for $h.$ as the result has already been proved for $f.$ in \cite[Cor.\,14]{knopf}. Therefore let $B,H\in\BB$ be arbitrary and let $C>0$ denote some generic constant depending only on $\mathring f$, $f_d$, $T$, $K$, $\beta$ and $\|\chi\|_{C^2_b}$. According to \cite[Lem.\,3]{knopf} Lemma there exist sequences $(B_k),(H_k)\subset C\big([0,T];C^2_b\big)$ such that 
	\begin{align*}
	\|B_k\|_\VV\le 2K,\quad \|H_k\|_\VV \le 2K, \quad \|B_k - B\|_{\VV} \to 0, \quad \|H_k - H\|_{\VV} \to 0
	\end{align*}
	if $k\to\infty$. By Corollary \ref{WSLVL}\,(b) (and its proof) the induced strong solutions $h_{B_k}$ and $h_{H_k}$ satisfy\vspace{-2mm}
	\begin{gather*}
	h_{B_k}\to h_B ,\; h_{H_k}\to h_H \quad\text{in}\; W^{1,2}(0,T;C_b) \cap C([0,T];C^1_b),\\
	\|h_{B_k}\|_{W^{1,2}(0,T;C_b)} + \|h_{B_k}\|_{C([0,T];C^1_b)} \le C,\\ \|h_{H_k}\|_{W^{1,2}(0,T;C_b)}+ \|h_{H_k}\|_{C([0,T];C^1_b)} \le  C.
	\end{gather*}
	The constant $C$ does not depend on $k$ since $\|B_k\|_{\VV}$ and $\|H_k\|_{\VV}$ are bounded by $2K$. Also note that there exists some constant $\varrho>0$ depending only on $\mathring f,f_d,T,K$ and $\beta$ (but not on $k$) such that ${\supp h_{B_k} \subset B_{\varrho}(0)}$, $\supp h_{H_k} \subset B_{\varrho}(0)$. As $h_{B_k}$ and $h_{H_k}$ are classical solutions they satisfy the implicit representation formula \eqref{EXPLVL2}. We also know from \cite[Lem.\,9]{knopf} (with $\varrho$ instead of $R$) that
	\begin{gather*}
	\|f_{B_k}(t) - f_{H_k}(t)\|_\infty \le C\; \|B_k - H_k\|_{\VV}\;,\\
	\|D_z f_{B_k}(t) - D_z f_{H_k}(t)\|_\infty \le C\; \|B_k - H_k\|_{\VV}^\gamma\;,\\
	\|Z_{B_k}(t) - Z_{H_k}(t)\|_{L^\infty(B_\varrho(0))} \le C\; \|B_k - H_k\|_{\VV}\\
	\|D_z Z_{B_k}(t) - D_z Z_{H_k}(t)\|_{L^\infty(B_\varrho(0))}\le C\; \|B_k - H_k\|_{\VV}^\gamma
	\end{gather*}
	for all $t\in[0,T]$. Together with \cite[Lem.\,8]{knopf} this yields
	\begin{align*}
	&\|h_{B_k}(t) - h_{H_k}(t)\|_{L^\infty} \\[2mm]
	&\le \|f_d\|_{C^1_b} \;\|Z_{B_k}(T,t,\cdot)-Z_{H_k}(T,t,\cdot)\|_{L^\infty(B_{\varrho}(0))}\\
	&\quad + \int\limits_t^T \|\Phi_{f_{B_k},h_{B_k}}(s,Z_{B_k}(s,t,\cdot)) - \Phi_{f_{H_k},h_{H_k}}(s,Z_{H_k}(s,t,\cdot))\|_{L^\infty(B_{\varrho}(0))}\ds\\
	%
	%	&\le C\; \|B_k-H_k\|_{\VV}\\
	%	&\quad + C\; \int\limits_t^T \|h_{B_k}(s)\|_{L^\infty}\; \|D_z f_{B_k}(s)\|_{L\infty}\; \|Z_{B_k}(s)-Z_{H_k}(s)\|_{L^\infty(B_{\varrho}(0))} \ds\\
	%	&\quad + C\; \int\limits_t^T \|h_{B_k}(s)\|_{L^2}\; \|\delv f_{B_k}(s)- \delv f_{H_k}(s)\|_{L^\infty}\; \ds\\
	%	&\quad + C\; \int\limits_t^T \|h_{B_k}(s)- h_{B_k}(s)\|_{L^2}\; \|\delv f_{H_k}(s)\|_{L^\infty}  \ds \\
	%	%
	&\le C\; \|B_k-H_k\|_{\VV}^{\gamma} + C\; \int\limits_t^T \|h_{B_k}(s)- h_{B_k}(s)\|_{L^2}\ds
	\end{align*}
	and hence $\|h_{B_k} - h_{H_k}\|_{C([0,T];C_b)} \le C\; \|B_k-H_k\|_{\VV}^{\gamma}$. By a similar computation,
	\begin{align*}
	\|\delz h_{B_k}(t) - \delz h_{H_k}(t)\|_{L^\infty} 
	\le C \|B_k-H_k\|_{\VV}^{\gamma} + C\int\limits_t^T \|\delz h_{B_k}(s)- \delz h_{B_k}(s)\|_{L^2}\ds
	\end{align*}
	and consequently $\|\delz h_{B_k}(t) - \delz h_{H_k}\|_{C([0,T];C_b)} \le C\; \|B_k-H_k\|_{\VV}^{\gamma}$ by Gronwall's lemma. 
	Expressing $\delt h_{B_k}$ and $\delt h_{H_k}$ by their corresponding Vlasov equation then yields $\|\delt h_{B_k} - \delt h_{H_k}\|_{L^2(0,T;C_b)} \le C\; \|B_k-H_k\|_{\VV}^{\gamma}$. In summary, we have established that
	\begin{align*}
	\|h_{B_k} - h_{H_k}\|_{W^{1,2}(0,T;C_b)} + \|h_{B_k} - h_{H_k}\|_{C([0,T];C^1_b)}\le C\; \|B_k-H_k\|_{\VV}^{\gamma}\;.
	\end{align*}
	For $k\to\infty$ this directly implies that
	\begin{align*}
	\|h_{B} - h_{H}\|_{W^{1,2}(0,T;C_b)} + \|h_{B} - h_{H}\|_{C([0,T];C^1_b)}\le C\; \|B-H\|_{\VV}^{\gamma}\;.
	\end{align*}
	and hence
	\begin{align*}
	\|g_{B} - g_{H}\|_{W^{1,2}(0,T;C_b)} + \|g_{B} - g_{H}\|_{C([0,T];C^1_b)}\le C\; \|B-H\|_{\VV}^{\gamma}\;.
	\end{align*}
	\smallskip
	\textit{Step 3}: We must still prove that $g_B\in L^\infty(0,T;H^2)$. Since $f_B\in L^\infty(0,T;H^2)$ has already been established in \cite[Thm.\,12]{knopf}, it suffices to show that $h_B$ is twice weakly differentiable with respect to $z$ and $D^2_z h_B\in L^\infty(0,T;L^2)$. Recall that for any $k\in\NN$, $f_{B_k}\in C([0,T];C^2_b)$ according to \cite[Thm.\,7]{knopf} and $h_{B_k}\in C([0,T];C^1_b)$ according to Theorem \ref{CSLVL}. Thus for all $i\in\{1,2,3\}$,
	\begin{gather*}
	\delxi\delvi f_{B_k}\, h_{B_k} + \delvi f_{B_k}\, \delxi h_{B_k} \in C([0,T];C_b),\\
	\supp \big[\delxi\delvi f_{B_k}\, h_{B_k} + \delvi f_{B_k}\, \delxi h_{B_k}\big](t) \subset \BR, \quad \text{for all}\; t\in [0,T],\\
	\psi_{(\delxi\delvi f_{B_k}\, h_{B_k} + \delvi f_{B_k}\, \delxi h_{B_k})} \in C\big([0,T];C^1_b\big) \cap C\big([0,T];H^2(B_{r}(0))\big)\,,\quad r>0.
	\end{gather*}
	The third line follows from \cite[Lem.\,2]{knopf}\vspace{-2mm}. Consequently,
	\begin{align*}
	&\Phi_{f_{B_k},h_{B_k}} = \sum_{i=1}^3 \delxi\psi_{\delvi f_{B_k}\,h_{B_k}} = \sum_{i=1}^3 \psi_{(\delxi\delvi f_{B_k}\, h_{B_k} + \delvi f_{B_k}\, \delxi h_{B_k})}
	\end{align*}
	lies in $C\big([0,T];C^1_b\big) \cap C\big([0,T];H^2(B_{r}(0))\big)$ for any $r>0$ and hence
	\begin{align}
	\label{EQZBK}
	\Phi_{f_{B_k},h_{B_k}}\,\chi \in C\big([0,T];C^1_b\big) \cap C\big([0,T];H^2\big)
	\end{align}
	since $\chi$ is compactly supported. We also know from \cite[Lem.\,8]{knopf} (with $\varrho$ instead of $R$) that $Z_{B_k}$ is twice continuously differentiable with respect to $z$ and
	\begin{align*}
	\big\|  t\mapsto Z_{B_k}(s,t,\cdot)  \big \|_{L^\infty(0,T;H^2(B_\varrho(0)))} \le C,\quad s\in[0,T].
	\end{align*}
	Now recall the implicit representation formula \eqref{EXPLVL2} for $h_{B_k}$ that is\vspace{-1mm}
	\begin{align}
	\label{EQHBK}
	h_{B_k}(t,z) = f_d\big(Z_{B_k}(T,t,z)\big) - \int\limits_t^T \big[ \Phi_{f_{B_k},h_{B_k}}\, \chi \big]\big( s,Z_{B_k}(s,t,z) \big) \ds
	\end{align}
	for all $(t,z) \in [0,T]\times \RR6$. As $f_d\in C^2_c(\RR^6)$ and $Z_{B_k}(T,t,\cdot)\in C^2(\RR^6)$, the term $f_d(Z_{B_k}(T,t,z))$ is twice continuously differentiable with respect to $z$ by chain rule. 
	By approximating $\Phi_{f_{B_k},h_{B_k}}\,\chi$ by sufficiently smooth functions one can easily show that the second summand of \eqref{EQHBK} is twice weakly differentiable and the derivatives can be computed by chain rule (with weak instead of classical derivatives if necessary). 
	\newpage\noindent
	Thus, for any $i,j\in\{1,2,3\}$, the weak derivative $\delzi\delzj h_{B_k}$ can be bounded by
	\begin{align*}
	\|\delzi\delzj h_{B_k}(t)\|_{L^2}
	&\le C\,\|f_d\|_{C^2_b}\; \|Z_{B_k}(0,t,\cdot) \|_{H^2(B_{\varrho}(0))}\\
	&\quad + C\,\int\limits_t^T \|\Phi_{f_{B_k},h_{B_k}}(s)\|_{H^2(B_{\varrho}(0))}\; \|\delz Z_{B_k}(s,t,\cdot)\|_{L^\infty(B_{\varrho}(0))}^2 \ds\\
	&\quad + C\,\int\limits_t^T \|\Phi_{f_{B_k},h_{B_k}}(s)\|_{W^{1,\infty}(B_{\varrho}(0))}\; \|Z_{B_k}(s,t,\cdot)\|_{H^2(B_{\varrho}(0))} \ds\\%[2mm]
	&\le C + C\,\|Z_{B_k}(0) \|_{H^2(B_{\varrho}(0))} + C\,\int\limits_0^T  \|Z_{B_k}(s)\|_{H^2(B_{\varrho}(0))} \ds\;.
	\end{align*}
	By \eqref{EQZBK} this finally yields $\|\delzi\delzj h_{B_k}\|_{L^\infty(0,T;L^2)}^2 \le C$. Then $(\delzi\delzj h_{B_k})$ is converging with respect to the weak-*-topology on $[L^1(0,T;L^2)]^*=L^\infty(0,T;L^2)$ up to a subsequence. Because of uniqueness, the weak-*-limit of the sequence $(\delzi\delzj h_{B_k})$ must be $\delzi\delzj h_{B}$ and especially $h_B \in L^\infty(0,T;H^2)$. This completes the proof.
\end{proof}

Now inserting the state $f_B$ and its costate $g_B$ in \eqref{DJLAG} yields
\begin{align}
\label{EQFDJ}
J'(B)[H] = (\partial_B \Lag)(f_B,B,g_B)\big[H\big],\quad H\in\VV
\end{align}
since $f_B'[H]\big\vert_{t=0} = 0$. This provides a necessary optimality condition:

\begin{thm}
	\label{NOC1}
	\hypertarget{HNOC1} $\;$
	\begin{itemize}
		\item[\textnormal{(a)}]
		The Fréchet derivative of $J$ at the point $B\in\BB$ is given by 
		\begin{align*}
		J'(B)[H] = 	\hspace{-5pt}\int\limits_{[0,T]\times\RR^3}\hspace{-5pt} \left( -\lambda \laplace_x B + \int\limits_{\RR^3} v\times \delv f_B \; g_B \dv \right)\cdot H \dtx, \quad H\in\VV.
		\end{align*}
		%for all $H\in\VV$.
		%
		\item[\textnormal{(b)}]
		Let us assume that $\B*\in\BB$ is a locally optimal solution of the optimization problem \eqref{OP1}. Then for all $B\in\BB$,\vspace{-1mm}
		\begin{align*}
		\int\limits_{[0,T]\times\RR^3} \hspace{-7pt} \left( -\lambda \laplace_x \B* + \int\limits_{\RR^3} v\stimes \delv f_\B* \; g_\B* \dv \right)\hspace{-2pt}\cdot\hspace{-1pt} (B-\bar B) \; \mathrm d(t,x)
		\begin{cases}
		= 0, &\hspace{-5pt} \text{if } \B* \in \IBB \\
		\ge 0, &\hspace{-5pt} \text{if } \B* \in \partial\BB
		\end{cases}\hspace{-5pt}.
		\end{align*}
		\item[\textnormal{(c)}]
		If we additionally assume that $\B*\in\IBB$ then $\B*$ satisfies the semilinear Poisson equation
		\begin{align}
		\label{FRJ}
		-\laplace_x \B* = -\frac 1 \lambda \int\limits_{\RR^3} v\times \delv f_\B* \; g_\B* \dv
		.	\end{align}
		In this case $\B*\in C([0,T];C^2_b(\RR^3))$ with
		\begin{align}
		\label{FRJ-2}
		\B*(t,x) = -\frac 1 {4\pi\lambda} \iint \frac{1}{|x-y|}\; w\times \delv f_\B*(t,y,w)\; g_\B*(t,y,w) \;\mathrm dw \mathrm dy
		\end{align}
		for all $t\in[0,T]$ and $x\in\RR^3$. Thus $\B*$ does not depend on the choice of $\chi$ as long as $\chi = 1$ on $\BRZ$ as it only depends on $g_\B*\big\vert_\BR$. 
	\end{itemize}
\end{thm}

\begin{proof} (a) follows immediately from \eqref{DBLAG} and \eqref{EQFDJ}. (b) is a direct consequence of Lemma \ref{FDJ-1} and (a) with ${H:=B-\bar B}$ and (b) implies \eqref{FRJ}. Recall that for almost all $t\in [0,T]$, $\B*(t)$ has a continuous representative satisfying $\B*_i(t,x) \to 0$ if $|x|\to\infty$ for every $i\in\{1,2,3\}$. Hence $\B*$ is uniquely determined by \eqref{FRJ-2}.
	We must still prove that $\B*$ lies in $C([0,T];C^2_b(\RR^3))$. Recall that $f_\B*$ and $g_\B*$ are in $ W^{1,2}(0,T;C_b(\RR^6))$ $\cap\, C([0,T];C_b^1(\RR^6))$ as $\B*\in\BB$. Thus
	\begin{align*}
	p:[0,T]\times\RR^3 \to\RR^3,\quad (t,x)\mapsto \int\limits v\times \delv f_\B* \; g_\B* \dv
	\end{align*}
	is continuous with $\supp p(t)\subset \BR$ for all $t\in[0,T]$. By approximating $f_\B*$ by $C([0,T];C^2_b)$-functions and using integration by parts one can easily show that $p$ is continuously differentiable where the partial derivatives are given by
	\begin{align*}
	\delxi p =  - \int\limits (v\times \delv f_\B*) \; \delxi g_\B* - (v\times \delv g_\B*) \; \delxi f_\B*\dv, \quad i=1,2,3.
	\end{align*}
	Consequently $\B*\in C([0,T];C^2_b(\RR^3;\RR^3))$. Since $g_\B*$ does not depend on $\chi$ as long as $\chi=1$ on $\BRZ$ the same holds for $\B*$. \end{proof}

Note that Theorem \ref{NOC1} provides only a necessary but not a sufficient condition for local optimality. If a control $B$ satisfies the above condition it could still be a saddle point or even a local maximum point. Theorem \ref{NOC1} does also not provide uniqueness of the locally optimal solution. However the globally optimal solution that is predicted by \cite[Thm.\,16]{knopf} is also locally optimal. Thus we have at least one control to satisfy the necessary optimality condition of Theorem \ref{NOC1}. \pskip

Assuming that there exists a locally optimal solution $\B*\in\IBB$ we can easily deduce from Theorem \ref{NOC1} that the triple $(f_\B*,g_\B*,\B*)$ is a classical solution of some certain system of equations.

\begin{cor}
	\label{OSY1}
	Suppose that $\B*\in\IBB$ is a locally optimal solution of the optimization problem \eqref{OP1}. Let $f_\B*$ and $g_\B*$ be its induced state and costate. Then $f_\B*, g_\B*\in C^1([0,T]\times\RR^6)$  and the triple
	$(f_\B*,g_\B*,\B*)$ is a classical solution of the \textbf{optimality system}
	\begin{align}
	\label{OS-SEQ0}
	\begin{cases}
	\delt f + v\cdot \delx f - \delx\psi_f\cdot\delv f + (v\stimes B)\cdot\delv f = 0, &\hspace{-40pt} f\big\vert_{t=0} = \mathring f\\[0.15cm]
	\delt g + v\cdot \delx g - \delx\psi_f\cdot\delv g + (v\stimes B)\cdot\delv g = \Phi_{f,g}\chi, &\hspace{-40pt} g\big\vert_{t=T} = f(T)-f_d\\[0.15cm]
	B(t,x) = -\frac 1 {4\pi\lambda} \iint \frac{1}{|x-y|}\; w\times \delv f(t,y,w)\; g(t,y,w) \dyw\;.
	\end{cases}
	\end{align}
	For all $t\in [0,T]$, $\supp f_\B*(t) \subset \BR$ and $\supp g_{\B*}(t) \subset B_{R^*}(0)$.
\end{cor}

\begin{proof}
	From Theorem \ref{NOC1} we know that $\B*\in C([0,T];C^2_b)$. Thus by \cite[Thm.\,7]{knopf} the solution $f_\B*$ is classical, lies in $C^1([0,T]\times\RR^6)\cap C([0,T];C^2_b)$ and satisfies $\supp f_\B*(t) \subset \BR$, $t\in[0,T]$. We can use the decomposition $g_\B*=f_\B*+h_\B*$ from the proof of Theorem \ref{ADJS} and from Proposition \ref{CSLVL} we can easily deduce that $g_\B*$ is classical, i.e., $g_\B*\in C^1([0,T]\times\RR^6)$ with $\supp g_\B*(t) \subset B_{R^*}(0)$, $t\in[0,T]$. The rest is obvious due to the construction of $f_\B*$, $g_\B*$ and Theorem \ref{NOC1}.
\end{proof}

\subsection{A sufficient condition for local optimality}

%To motivate the following approach let us at first consider the following example: Suppose that $\varphi:\RR^d \to\RR$ is a twice continuously differentiable function and  let $U$ be a convex open subset of $\RR^d$. Now, if there exists some point $ x\in \bar U$ such that $\grad \varphi(x)\cdot h \ge 0$ for all $h\in\RR^d$ with $x+h\in\bar U$ and $D_x^2\varphi( x)$ is strictly positive definit, we can conclude that $ x$ is a strict local minimum of $\varphi$.
%Again, this fact can be generalized to functionals on Banach spaces using the Fréchet derivatives of first and second order.

To prove that our cost functional is twice continuously Fréchet differentiable we will need Fréchet differentiability of first order of the costate.

\begin{lem}
	\label{FDCCSO}
	\hypertarget{HFDCCSO}
	Let $g.:\BB\to C([0,T];L^2(\RR^6)),\; B \mapsto g_B$ denote the field-costate operator. For any $B\in\BB$ and $H\in\VV$ there exists a unique strong solution $g_B^H\in {H^1(]0,T[\times\RR^6)}$ of the final value problem
	\begin{displaymath}
	\refstepcounter{equation}
	\label{GDEQ}
	\left\{
	\begin{aligned}
	&\delt g + v\cdot\delx g - \delx\psi_{f'_B[H]}\cdot\delv g_B - \delx\psi_{f_B}\cdot\delv g + (v\stimes B)\cdot\delv g + (v\stimes H)\cdot\delv g_B \\
	&\qquad = \Phi_{f_B,g}\chi - \Phi_{g_B,f_B'[H]}\chi \\[0.2cm]
	&g\big\vert_{t=T}=0\;. 
	\end{aligned} 
	\right.\hspace{-17pt} \textnormal{(\theequation)}
	\end{displaymath}
	Then the following holds:
	\begin{itemize}
		\item [\textnormal{(a)}] The control-costate operator $g.$ is Fr\'echet differentiable on $\IBB$ with respect to the $C([0,T];L^2(\RR^6))$-norm, i.e., for any $B\in\IBB$ there exists a unique linear operator $g'_B: \VV\to C([0,T];L^2(\RR^6))$ such that\\
		$$\forall \eps>0\; \exists \delta > 0\; \forall H\in \VV \text{ with } \|H\|_{\VV} < \delta :$$
		$$B+H \in \IBB \tand \frac{\| g_{B+H} - g_B - g'_B[H] \|_{C([0,T];L^2)}}{\|H\|_{\VV}} < \eps.$$
		The Fr\'echet derivative is given by $g'_B[H]=g_B^H$ for all $H\in\VV$.
		\item [\textnormal{(b)}] For all $B,H\in\IBB$, the solution $g_B^H$ depends Hölder-continuously on $B$ in such a way that there exists some constant $C>0$ depending only on $\mathring f, T, K$ and $\beta$ such that for all $A,B\in\IBB$,
		\begin{align}
		\label{CFG}
		\underset{\|H\|_\VV\le 1}{\sup}\;\|g_A'[H] - g_B'[H]\|_{L^2(0,T;L^2)} \le C \;\|A-B\|_{\VV}^{\gamma}.
		\end{align}
	\end{itemize}
\end{lem}

The proof proceeds analogously to the proof of Theorem \ref{FDCSO}.

\begin{com}
	As $K$ was arbitrary the above results hold true if $\IBB$ is replaced by~$\IBBB$. Hence they are especially true on $\BB$.\\
\end{com}

Continuous differentiability of the cost functional then follows:

\begin{cor}
	\label{TFDJ}
	\hypertarget{HTFDJ}
	The cost functional $J$ of the optimization problem \eqref{OP1} is twice Fréchet differentiable on $\IBB$. The Fréchet derivative of second order at the point $B\in\IBB$ can be described as a bilinear operator $J''(B):\VV^2\to \RR$ that is given by
	\begin{align*}
	J''(B)[H_1,H_2] &= \lambda\; \langle D_x H_1, D_x H_2 \rangle_{L^2([0,T]\times\RR^3)}\\
	&\quad - \int\limits_{[0,T]\times\RR^6} (v\times H_1)\cdot \big( \delv f_B\; g_B'[H_2] - \delv g_B\; f_B'[H_2] \big) \; \mathrm d(t,x,v)
	\end{align*}
	for all $H_1,H_2\in\IBB$. Moreover there exists some constant $C>0$ depending only on $\mathring f$, $f_d$, $T$, $K$ and $\beta$ such that for all $B,\tB\in\IBB$,
	\begin{align*}
	\| J''(B) - J''(\tB) \| \le C\, \|B-\tB\|_{\VV}^\gamma
	\end{align*}
	where
	\begin{align*}
	\| J''(B) \| = \sup\Big\{ \big|J''(B)[H_1,H_2] \big| \,\Big\vert\, \|H_1\|_{\VV} = 1 ,\, \|H_2\|_{\VV} = 1\Big\}
	\end{align*}
	denotes the operator norm. This means that $J$ is twice continuously differentiable.
\end{cor}

%\smallskip
%
%\begin{com} By definition the Fréchet derivative of second order is the Fréchet derivative of the Fréchet derivative of first order. This means that, in the proper sense, it is an operator 
%	$J''(B):\VV\to \mathcal L\left(\,\VV;\,\mathcal L\big(\VV;\RR\big)\,\right)$. Because of the two linear dependences we can equivalently consider the Fréchet derivative of second order as a bilinear operator $J''(B):\VV^2\to \RR$ as it was done in the above proposition. Since $K$ was arbitrary, $\IBB$ can be replaced by $\IBBB$ and hence all results of this proposition are also true on $\BB$ instead of $\IBB$.
%\end{com}
%
%\smallskip

\begin{proof} Theorem \ref{FDCSO} and Theorem \ref{FDCCSO} provide the decompositions\vspace{-1mm}
	\begin{align*}
	f_{B+H}-f_B = f'_B[H] + f_R[H],\qquad g_{B+H}-g_B = g'_B[H] + g_R[H]
	\end{align*}
	for $B\in\BB$, $H\in\VV$ with $B+H\in\BB$ where
	\begin{align*}
	\|f_R[H]\|_{C([0,T];L^2)} = \text{o}(\|H\|_{\VV}),\qquad  \|g_R[H]\|_{C([0,T];L^2)} = \text{o}(\|H\|_{\VV})
	\end{align*}
	if $\|H\|_{\VV}$ tends to zero. Let now $B\in\BB$ and $H_1,H_2\in\VV$ with $B+H_2\in\BB$ be arbitrary . Then, by Theorem \ref{NOC1}\,(a),
	\begin{align*}
	&J'(B+H_2)[H_1]-J'(B)[H_1] \\[0.2cm]
	%&\quad = \lambda\langle D_x H_1,D_x H_2\rangle_{L^2} - \hspace{-10pt}\int\limits_{[0,T]\times\RR^6}\hspace{-8pt} (v\stimes H_1)\cdot\big(\delv f_{B+H_2} \; g_{B+H_2} - \delv f_B \; g_B \big) \dtxv \\[0.2cm]
	&\quad = \lambda\langle D_x H_1,D_x H_2\rangle_{L^2} \\
	&\qquad - \int\limits_{[0,T]\times\RR^6} (v\times H_1)\cdot \big( \delv f_B \; g'_B[H_2] - \delv g_B \; f'_B[H_2] \big)\dtxv + \mathcal R 
	\end{align*}
	where
	\begin{align*}
	\mathcal R&:= - \hspace{-0.2cm}\int\limits_{[0,T]\times\RR^6} (v\times H_1)\cdot \big( \delv f_B \; g_R[H_2] - \delv g_B \; f_R[H_2] \big)\; \mathrm d(t,x,v)\\
	&\qquad - \hspace{-0.2cm}\int\limits_{[0,T]\times\RR^6} (v\times H_1)\cdot  (\delv f_{B+H_2}-\delv f_B ) \; (g_{B+H_2} - g_B) \; \mathrm d(t,x,v).
	\end{align*}
	Using \eqref{HCF}, \eqref{HCG}, \eqref{CFF} and \eqref{CFG} one can easily show that $$\|R\| = \sup\big\{ |R| :\|H_1\|_\VV \le 1 \big\}= o(\|H_2\|_{\VV})$$ and hence $J$ is twice Fréchet differentiable at the point $B$ and the Fréchet derivative is given by\vspace{-0.2cm}
	\begin{align*}
	J''(B)[H_1,H_2] &= \lambda\; \langle D_x H_1, D_x H_2 \rangle_{L^2([0,T]\times\RR^3)}\\
	&\quad - \int\limits_{[0,T]\times\RR^6} (v\times H_1)\cdot \big( \delv f_B\; g_B'[H_2] - \delv g_B\; f_B'[H_2] \big) \; \mathrm d(t,x,v)
	\end{align*}
	for all $H_1,H_2\in\VV$. To prove continuity let $B,\tB\in\BB$ and $H_1,H_2\in\VV$ be arbitrary and suppose that $\|H_i\|_\VV\le 1$ for $i=1,2$. Then
	\begin{align}
	\label{ESD2J}
	&|J''(B)[H_1,H_2] - J''(\tB)[H_1,H_2]| \notag\\
	&\quad = \Bigg|\; \int\limits_0^T \int (v\times H_1)\cdot \Big( \delv f_B\; g_B'[H_2] - \delv f_\tB\; g_\tB'[H_2] \notag \\[-4mm]
	&\hspace{4cm}- \delv g_B\; f_B'[H_2] +  \delv g_\tB\; f_\tB'[H_2] \Big)\dz\mathrm dt \;\Bigg|  \notag\\
	&\quad \le C\; \int\limits_0^T \|H_1(t)\|_{\infty} \; \Big[ \|\delv f_B(t) - \delv f_\tB(t)\|_{L^2} \; \|g_\tB'[H_2](t)\|_{L^2} \notag\\[-4mm]
	&\hspace{4cm} + \|\delv f_B(t)\|_{L^2} \; \|g_B'[H_2](t)-g_\tB'[H_2](t)\|_{L^2}\notag \\
	&\hspace{4cm} + \|\delv g_B(t) - \delv g_\tB(t)\|_{L^2}\; \|f_\tB'[H_2](t)\|_{L^2}\notag \\
	&\hspace{4cm} + \|\delv g_B(t)\|_{L^2} \; \|f_B'[H_2](t) - f_\tB'[H_2](t)\|_{L^2} \Big] \dt\notag\\
	&\quad \le C \; \Big[ \|g_\tB'[H_2]\|_{L^2(0,T;L^2)} + \|\delv f_B\|_{C(0,T;L^2)} + \|f_\tB'[H_2]\|_{L^2(0,T;L^2)} \notag\\
	&\qquad\qquad + \|\delv g_B(t)\|_{C(0,T;L^2)}\Big] \; \|B-\tB\|_{\VV}^{\gamma}  \notag\\[0.2cm]
	&\quad \le C\,\|B-\tB\|_{\VV}^{\gamma}
	\end{align}
	where the constant $C>0$ depends only on $\mathring f$, $f_d$, $T$, $K$ and $\beta$. This directly yields continuity of the second order derivative with respect to the operator norm. 
\end{proof}

The following theorem provides a sufficient condition for local optimality:

\begin{thm}
	\label{QGC1}
	\hypertarget{HQGC1}
	Suppose that $\B* \in\BB$ and let $f_\B*$ and $g_\B*$ be its induced state and costate. Let $0<\alpha<2+\gamma$ be any real number. We assume that the variation inequality
	\begin{align}
	\label{VIQ}
	\int\limits_{[0,T]\times\RR^3} \hspace{-0.2cm}\left( -\lambda \laplace_x \B* + \int\limits_{\RR^3} v\times \delv f_\B* \; g_\B* \dv \right)\cdot (B-\B*) \; \mathrm d(t,x) =\; J'(\B*)[B-\B*]\; \ge 0
	\end{align}
	holds for all ${B\in\BB}$ and that there exists some constant $\eps>0$ such that 
	\begin{displaymath}
	\refstepcounter{equation}
	\begin{aligned}
	\label{CIQ}
	&\lambda\; \|D_x H\|_{L^2([0,T]\times\RR^3)}^2
	- \hspace{-0.2cm}\int\limits_{[0,T]\times\RR^6} (v\times H)\cdot \big( \delv f_\B*\; g_\B*'[H] - \delv g_\B*\; f_\B*'[H] \big) \; \mathrm d(t,x,v)\\[0.3cm]
	&\quad = \; J''(\B*)[H,H] \;  \ge \eps \; \|H\|_{\VV}^\alpha \hspace{195pt} \textnormal{(\theequation)}
	\end{aligned}
	\end{displaymath}
	holds for all $H\in\BB$. Then $J$ satisfies the following growth condition: There exists $\delta>0$ such that for all $B\in\BB$ with $\|B-\B*\|_{\VV}<\delta$,
	\begin{align}
	\label{QGC}
	J(B) \ge J(\B*) + \frac \eps 4 \|B-\B*\|_{\VV}^\alpha
	\end{align}
	and hence $\B*$ is a strict local minimizer of $J$ on the set $\BB$.
\end{thm}

\begin{proof} Let $B\in\BB$ be arbitrary. We define the auxillary function \linebreak $F:[0,1]\to\RR_0^+$, ${s \mapsto J\big(\B*+s(B-\B*)\big)}$. Then $F$ is twice continuously differentiable by chain rule and Taylor expansion yields $F(1) = F(0) + F'(0) + \tfrac 1 2 F''(\vartheta)$ for some $\vartheta\in ]0,1[$. By the definition of $F$ this implies that
	\begin{align*}
	J\big(B\big) &= J\big(\B*\big) + J'\big(\B*\big)[B-\B*] + \tfrac 1 2 J''\big(\B*+\vartheta (B-\B*)\big)[B-\B*,B-\B*] \\[2mm]
	&\ge  J\big(\B*\big) +  \tfrac 1 2 J''\big(\B*\big)[B-\B*,B-\B*] \\
	&\quad +  \tfrac 1 2 \Big( J''\big(\B*+\vartheta (B-\B*)\big) - J''\big(\B*\big) \Big)[B-\B*,B-\B*]
	\end{align*}
	Now, according to Corollary \ref{TFDJ},
	\begin{align*}
	&\Big| \Big( J''\big(\B*+\vartheta (B-\B*)\big) - J''\big(\B*\big) \Big)[B-\B*,B-\B*] \Big| \le C\, \|B-\B*\|_{\VV}^{2+\gamma}
	%	&\quad \le \big\|J''\big(\B*+\vartheta (B-\B*)\big) - J''\big(\B*\big) \big\|\, \|B-\B*\|_{\VV}^2\\
	\end{align*}
	Suppose now that $ \|B-\B*\|_{\VV}<\delta$ for some $\delta>0$. Then
	\begin{align*}
	&\Big| \Big( J''\big(\B*+\vartheta (B-\B*)\big) - J''\big(\B*\big) \Big)[B-\B*,B-\B*] \Big| \\
	&\quad \le \; C\, \delta^{2+\gamma-\alpha} \|B-\B*\|_{\VV}^\alpha \;\le\; \frac{\eps}{2} \|B-\B*\|_{\VV}^\alpha
	\end{align*}
	if $\delta$ is sufficiently small. In this case $J\big(B\big) \ge J\big(\B*\big) + (\eps/4) \|B-\B*\|_{\VV}^\alpha$. This especially means that $J(B)>J(\B*)$ for all $B\in B_\delta (\B*) \cap \BB$ and consequently $\B*$ is a strict local minimizer of $J$. \end{proof}

\subsection{Uniqueness of the optimal solution on small time intervals}

We know from Corollary \ref{OSY1} that for any locally optimal solution $\B*\in\IBB$ the triple $(f_\B*,g_\B*,\B*)$ is a classical solution of the optimality system
\begin{align}
\label{OS1}
\left\{
\begin{aligned}
&\delt f +v\cdot \delx f - \delx\psi_f\cdot\delv f + (v\times B)\cdot\delv f = 0, &&\hspace{-2cm} f\big\vert_{t=0} = \mathring f\\[0.15cm]
&\delt g +v\cdot \delx g - \delx\psi_f\cdot\delv g + (v\times B)\cdot\delv g = \Phi_{f,g}, &&\hspace{-2cm} g\big\vert_{t=T} = f(T)-f_d\\[0.15cm]
&B(t,x) = -\frac 1 {4\pi\lambda} \iint \frac{1}{|x-y|}\; w\times \delv f(t,y,w)\; g(t,y,w) \dyw \;.\\[0.15cm]
\end{aligned}
\right.
\end{align}
The following theorem states that the solution of this system of equations is unique if the final time $T$ is small compared to $\lambda$. As we will have to adjust $\frac T \lambda$ it is necessary to assume that $0<\lambda\le \lambda_0$ for some constant $\lambda_0>0$. Of course large regularaization parameters $\lambda$ do not make sense in our model, so we will just assume that $\lambda_0 = 1$.

\begin{thm}
	\label{UOS}
	\hypertarget{HUOS}
	Suppose that $\lambda\in ]0,1]$ and suppose that there exists a classical solution $(f,g,B)$ of the optimality system \eqref{OS1}, i.e., $B\in C\big([0,T];C^1_b(\RR^3;\RR^3)\big)$ and  $f,g\in C^1([0,T]\times\RR^6)$ with $\supp f(t),\,\supp g(t) \subset B_r(0)$ for some radius $r>0$. Then this solution is unique if the quotient $\tfrac T \lambda$ is sufficiently small.
\end{thm}

\begin{proof}
	Suppose that the triple $(\tf,\tg,\tB)$ is another classical solution that is satisfying the support condition with radius $\tilde r$. Without loss of generality we assume that $r=\tilde r$. Let $C=C(T)\ge 0$ denote some generic constant that may depend on $T$, $\mathring f$, $f_d$, $r$, $\|\chi\|_{C^1_b}$ and the $C([0,T];C^1_b)$-norm of $f$, $\tf$, $g$ and $\tg$. We can assume that $C=C(T)$ is monotonically increasing in $T$. First of all, by integration by parts,
	\begin{align}
	\label{UNIEU}
	&\|B(t)-\tB(t)\|_\infty \le  \frac C \lambda \|g(t)-\tg(t)\|_\infty +  \frac C \lambda \|f(t)-\tf(t)\|_\infty, \quad t\in [0,T].
	\end{align}
	Let now $Z$ and $\tZ$ denote the solutions of the characteristic system of the Vlasov equation to the fields $B$ and $\tB$ satisfying $Z(t,t,z) = z$ and $\tZ(t,t,z)=z$ for any $t\in[0,T]$ and $z\in\RR^6$. Then for any $s,t\in [0,T]$ (where $s\le t$ without loss of generality) and $z\in\RR^6$,
	\begin{align*}
	&|Z(s,t,z)-\tZ(s,t,z)| \\
	&\quad \le \int\limits_s^t C\; |Z(\tau,t,z)-\tZ(\tau,t,z)| + C\; \|\delx\psi_{f-\tf}(\tau)\|_\infty + C\; \|B(\tau)-\tB(\tau)\|_\infty \dtau\\
	&\quad \le \int\limits_s^t C\; |Z(\tau,t,z)-\tZ(\tau,t,z)| + \tfrac C \lambda\; \|f(\tau)-\tf(\tau)\|_\infty + \tfrac C \lambda\; \|g(\tau)-\tg(\tau)\|_\infty \dtau
	\end{align*}
	and hence
	\begin{align}
	\label{EQ0}
	|Z(s,t,z)-\tZ(s,t,z)| \le C\; \int\limits_s^t \tfrac 1 \lambda\; \|f(\tau)-\tf(\tau)\|_\infty + \tfrac 1 \lambda\; \|g(\tau)-\tg(\tau)\|_\infty \dtau
	\end{align}
	by Gronwall's lemma. Consequently
	\begin{align*}
	\|f(t)-\tf(t)\|_\infty &\le C\; \|Z(0,t,\cdot)-\tZ(0,t,\cdot)\|_\infty \\
	&\le C\; \int\limits_0^t \tfrac 1 \lambda\; \|f(\tau)-\tf(\tau)\|_\infty + \tfrac 1 \lambda\; \|g(\tau)-\tg(\tau)\|_\infty \dtau
	\end{align*}
	which yields
	\begin{align*}
	\|f(t)-\tf(t)\|_\infty \le C\; \tfrac 1 \lambda \exp\left( C\;\tfrac T \lambda \right)\; \int\limits_0^t \|g(\tau)-\tg(\tau)\|_\infty \dtau
	\end{align*}
	and thus
	\begin{align}
	\label{UNIEF}
	\|f-\tf\|_{C([0,T];C_b)} \le C\; \tfrac T \lambda \exp\left( C\;\tfrac T \lambda \right)\; \|g-\tg\|_{C([0,T];C_b)}\,.
	\end{align}
	For $z\in B_r(0)$ and $t\in[0,T]$ we can conclude from \eqref{EXPLVL} that
	\begin{align*}
	&|g(t,z)-\tg(t,z)| \\[0.25cm]
	&\quad\le |\big(f(T)-f_d\big)(Z(T,t,z)) - \big(\tf(T)-f_d\big)(\tZ(T,t,z))|\\
	&\qquad + \int\limits_{t}^T |[\Phi_{f,g}\chi](\tau,Z(\tau,t,z)) - [\Phi_{\tf,\tg}\chi](\tau,\tilde Z(\tau,t,z))| \dtau \\
	%	&\quad\le C\; \|Z(T,t,\cdot)-\tZ(T,t,\cdot)\|_\infty \\
	%	&\qquad + \int\limits_{t}^T |\Phi_{f,g}(\tau,X(\tau,t,z)) - \Phi_{\tf,\tg}(\tau,X(\tau,t,z))| \dtau\\
	%	&\qquad + \int\limits_{t}^T |[\Phi_{\tf,\tg}\chi](\tau,Z(\tau,t,z)) - [\Phi_{\tf,\tg}\chi](\tau,\tilde Z(\tau,t,z))| \dtau\\
	&\quad\le C\; \|Z(T,t,\cdot)-\tZ(T,t,\cdot)\|_\infty + \int\limits_{t}^T \|\Phi_{f,g}(\tau) - \Phi_{\tf,\tg}(\tau)\|_{L^\infty(B_r(0))} \dtau\\
	&\qquad + C\;\int\limits_{t}^T \|\Phi_{\tf,\tg}(\tau)\|_{W^{1,\infty}} \|Z(\tau,t,\cdot) - \tilde Z(\tau,t,\cdot)\|_\infty \dtau\;.
	\end{align*}
	We already know from inequality \eqref{EQ0} that for $t\le\tau\le T$,
	\begin{align*}
	\|Z(\tau,t,\cdot) - \tilde Z(\tau,t,\cdot)\|_\infty  &\le C\; \int\limits_t^\tau \tfrac 1 \lambda\; \|f(\sigma)-\tf(\sigma)\|_\infty + \tfrac 1 \lambda\; \|g(\sigma)-\tg(\sigma)\|_\infty \;\mathrm d \sigma\,.
	\end{align*}
	Also recall that
	\begin{align*}
	&\|\Phi_{f,g}(\tau)\|_{W^{1,\infty}} \le \|\Phi_{f,g}(\tau)\|_\infty + \|\Phi'_{f,g}(\tau)\|_\infty \\
	&\quad \le C\; \|f\|_{C([0,T];C^1_b)}\, \|g\|_{C([0,T];C^1_b)} \le C
	\end{align*}
	for every $\tau\in[0,T]$. Moreover, by \eqref{ESTPHI3},
	\begin{align*}
	&\|\Phi_{f,g}(\tau)-\Phi_{\tf,\tg}(\tau)\|_{L^\infty(B_r(0))}\\ 
	&\quad \le C\,\|\delz \tg\|_\infty\, \|f(\tau)-\tf(\tau)\|_\infty + C\,\|\delz f\|_\infty\, \|g(\tau)-\tg(\tau)\|_\infty\\
	&\quad \le C\; \|f(\tau)-\tf(\tau)\|_\infty + C\; \|g(\tau)-\tg(\tau)\|_\infty
	\end{align*}
	for all $\tau\in[0,T]$. This implies that for all $t\in[0,T]$,
	\begin{align*}
	\|g(t)-\tg(t)\|_\infty
	&\le C\; \int\limits_t^T \tfrac 1 \lambda \; \|g(\tau)-\tg(\tau)\|_\infty + \tfrac 1 \lambda  \; \|f(\tau)-\tf(\tau)\|_\infty \dtau
	\end{align*}
	and hence
	\begin{align}
	\label{UNIEG}
	\|g-\tg\|_{C([0,T];C_b)} \le C\; \tfrac T \lambda \exp\left( C\;\tfrac T \lambda \right)\; \|f-\tf\|_{C([0,T];C_b)}
	\end{align}
	by Gronwall's lemma. Inserting \eqref{UNIEG} in \eqref{UNIEF} yields
	\begin{align*}
	\|f-\tf\|_{C([0,T];C_b)} \le C \left(\tfrac T \lambda\right)^2 \exp\left( C\;\tfrac T \lambda \right)\; \|f-\tf\|_{C([0,T];C_b)}\;.
	\end{align*}
	If now $\tfrac T \lambda$ is sufficiently small we have $C \left(\tfrac T \lambda\right)^2 \exp\left( C\;\tfrac T \lambda \right)<1$ and we can conclude that $f=\tf$ on $[0,T]\times\RR^6$. Then obviously $g=\tg$ by \eqref{UNIEG} and $B=\tB$ by \eqref{UNIEU} which means uniqueness of the solution $(f,g,B)$. \end{proof}

\bigskip

If $\B*\in\IBB$ is a locally optimal solution, the following uniqueness result holds:

\begin{cor}
	Suppose that $\lambda\in]0,1]$ and let $\B*\in\IBB$ be a locally optimal solution of the optimization problem \eqref{OP1}. Then the tripel $(f_{\B*},g_{\B*},\B*)$ is a classical solution of the optimality system \eqref{OS1} according to Corollary \ref{OSY1}.\pskip
	
	If now $\lambda\in]0,1]$ and $\tfrac T \lambda$ is sufficiently small then $\B*$ is the only locally optimal solution of the optimization problem \eqref{OP1} in $\IBB$. \pskip
	
	Suppose that there is a globally optimal solution $B\in\IBB$. Then $B=\B*$ is the unique globally optimal solution in $\IBB$. However it is still possible that there are other globally optimal solutions in $\partial\BB$.
\end{cor}

\begin{proof} If $\lambda\in]0,1]$ and $\tfrac T \lambda$ is sufficiently small then Proposition \ref{UOS} ensures that $\B*$ is the only locally optimal solution. Recall that there exists at least one globally optimal solution according to \cite[Thm.\,16]{knopf}. Let us assume that $B\in\IBB$ in one of these globally optimal solutions. As any globally optimal solution is also locally optimal it follows that there is only one globally optimal solution in $\IBB$ and thus $B=\B*$. \end{proof}

% Literaturverzeichnis
\bigskip
%\setlength{\bibsep}{0pt}
%\noindent\textbf{References}\\

\end{document}